\documentclass{article}
\usepackage[utf8]{inputenc}
\usepackage{amsmath}
\usepackage{mathtools}
\usepackage{amssymb}
\usepackage{amsthm}
\usepackage{hyperref}
\newtheorem{rem}{Remark}
\newtheorem{thm}{Theorem}[section]
\newtheorem{prop}[thm]{Proposition}
\newtheorem{prob}[thm]{Problem}
\newtheorem{cor}[thm]{Corollary}

\newtheorem{conj}[thm]{Conjecture}
\newtheorem{qn}[thm]{Question}
\newtheorem{lem}[thm]{Lemma}
\theoremstyle{definition}
\newtheorem{defn}[thm]{Definition}
\newcommand{\Sym}{\textrm{Sym}}
\newcommand{\supp}{\textrm{supp}}
\newcommand{\Spec}{\textrm{Spec}}
\newcommand{\Sing}{\textrm{Sing}}
\newcommand{\FF}{\mathbb{F}}
\newcommand{\LL}{\mathcal{L}}
\newcommand{\PP}{\mathbb{P}}
\newcommand{\OO}{\mathcal{O}}

\newcommand{\M}{\mathcal{M}}
\newcommand{\A}{\mathbb{A}}
\newcommand{\Lb}{\mathbb{L}}
\newcommand{\ZZ}{\mathbb{Z}}
\newcommand{\QQ}{\mathbb{Q}}
\newcommand{\CC}{\mathbb{C}}
\newcommand{\RR}{\mathbb{R}}
\newcommand{\Fcal}{\mathcal{F}}
\title{Error terms for the motives of discriminant complements and a Cayley-Bacharach theorem}
\author{Ishan Banerjee }

\begin{document}

\maketitle
\begin{abstract}
    In this paper we prove under some simplifying hypotheses questions of Picoco and Levinson-Ullery on Cayley-Bacharach sets. Our results imply that, under suitable hypotheses Cayley-Bacharach sets lie on curves of low degree. We then use these results to estimate error terms to the normalized motive of the space of smooth degree $d$ hypersurfaces in $\PP^n$ as $d$ grows to infinity. The error term can be expressed in terms of a certain `sum over points' on plane cubic curves and the associated Hodge structure can be expressed in terms of  the cohomology of the moduli space of elliptic curves. We also prove convergence of  the motive of degree $d$ hypersurfaces in $\PP^n$ as $n$ grows to infinity as well as other results on discriminant complements of high dimensional varieties.
\end{abstract}
\section{Introduction}
Let $X$ be a smooth projective variety over a field $\FF$. Let $\LL$ be an ample line bundle on $X$. Let $$\Sigma (X,\LL) := \{f \in H^0(X, \LL) | \exists p \in X, f(p) =0, df(p) =0\}.$$ We call $\Sigma(X, \LL)$ the discriminant variety associated to $X$ and $\LL$, it consists of those sections of $\LL$ which define a singular hypersurface in $X$. Let $$U(X, \LL) = H^0(X, \LL) \setminus \Sigma(X, \LL).$$ Theorems of Vakil-Wood and Poonen mply that after some suitable normalization, the motive of $U(X, \LL^j)$ converges (in an appropriate sense) as $j \to \infty$ (see \cite{VW} and \cite{Po}).

In this paper we study the \emph{rate of convergence} of the motive of $U(X, \LL^j)$.
Our main contributions in this regard are:
\begin{enumerate}
    \item An improvement of known rates of convergence for all $X$ and $\LL$ (Theorem \ref{mainrate}). Interestingly enough, our results suggest that the rate of convergence grows faster with the increase in the dimension of $X$, as in the upper bounds  we establish on the `error term' decrease as $\dim X \to \infty$.
    \item  In the special case of $\PP^n$, we obtain a leading error term  (see Theorem \ref{mainPn} ) for the convergence of $U(\PP^n, \OO(d))$ as $d \to \infty$.
    \item We establish upper bounds on the difference between the motive of $U(X, \OO(d))$ and a special value of the motivic zeta function of $X$ for all smooth projective varieties $X$ (see Theorem \ref{maindim}). Our bound depends only on the dimension of $X$. Our bounds improve as the dimension of $X$ increases.
    As a consequence we establish  (for $d \ge 3$ ) that after a suitable normalisation, the motive of $U(\PP^n, \OO(d))$ converges to a certain limit of motivic zeta functions as $n \to \infty$ (see Theorem \ref{maindim}). It is important to note here that we do not assume that $d \to \infty$ or take some kind of limit in which the line bundle $\LL$ becomes more and more ample.

\end{enumerate}
In order to establish the above results, we prove several theorems about Cayley Bacharach sets in $\PP^n$ which are of independent interest (the definition of Cayley-Bacharach sets will be recalled in Section 2). Our results regarding Cayley-Bacharach sets are as follows:
\begin{enumerate}
    \item We show that under certain hypotheses, Cayley-Bacharach subsets of $\PP^n$ are forced to lie on curves of low degree (see Theorem \ref{mainCB} ). This is partial progress on a question of Picoco(see Question 1.1 of \cite{Pi} )
    \item We establish Conjecture 1.2 of \cite{LU} under certain additional hypotheses (Theorem \ref{conjLU} ). The conjecture states that a  subset $Z$ of $\PP^n$ that is Cayley-Bacharach for $\OO(r)$, that satisfies $|Z|\le (d+1) r +1$ lies on a plane configuration of dimension $d$.
\end{enumerate}
\subsection{Grothendieck ring of varieties}
In this section, will briefly discuss the Grothendieck ring of varieties. We largely follow section 1 of \cite{VW} where the Grothendieck ring of varieties is discussed in more detail.

The Grothendieck ring  of varieties $\M$ over the field $\FF$ (also denoted $K_0(Var_{\FF})$) is defined  as follows, in terms of generators and relations.
\begin{enumerate}
    \item Given a finite type $\FF$ scheme $X$, we have an element $[X] \in \M$. Isomorphic schemes define the same element of $\M$.
    \item The collection of all $[X]$ generate $\M$ as an abelian group.
    \item Given any finite type scheme $X$  and a closed subscheme $Y$ with complement $U = X \setminus Y$, we have the relation $$[X] = [U] + [Y].$$
    \item Given $[X], [Y] \in \M$ we define $[X][Y] = [X \times Y]$. This gives $\M$ the structure of a ring.
\end{enumerate}

Let $\Lb = [\A^1].$ Let $\M_{\Lb} = \M [\Lb^{-1}]$. There is an increasing filtration on $\M$, with varieties of dimension $\le \ell$, generating the $\ell$th piece of the filtration. We will call this filtration the dimensional filtration on $\M$. The filtration  extends to a filtration on $\M_{\Lb}$. We denote the completion of $\M_{\Lb}$ with respect to this filtration by $\hat{\M_{\Lb}}$.

\begin{rem}
  In this paper we will often have formulae where an element $x \in \M$ is multiplied by a vector space over the base field $\FF$. These formulae are to be understood in the following sense- we identify a vector space $V$ with $\Lb^{\dim V}$. Any expression of the form $xV$ where $x \in \M$ is simply shorthand for  $x \Lb^{\dim V}$. Similarly an expression of the form $xV^{-1}$ is understood to mean $x \Lb^{- \dim V}$. 
\end{rem}

Let us now assume $\FF = \CC$. Let $\M^{Hdg}$ denote the Grothendieck group of mixed Hodge structures . Let $\Lb^{Hdg} = [H^2_c(\A^1, \QQ)]$. We define $\M_{\Lb}^{Hdg} = \M^{Hdg}[(\Lb^{Hdg})^{-1}].$ There is an increasing filtration on $\M^{Hdg}$ where the $\ell$th piece of the filtration is generated by Hodge structures of weight $\le \ell$. The filtration extends to a filtration on $\M_{\Lb}^{Hdg}$. We denote the completion of $\M_{\Lb}^{Hdg}$  with respect to this filtration $\hat{\M_{\Lb}}^{Hdg}$

 There is a specialisation map $ \phi: \M \to \M^{Hdg}$ defined on smooth proper generators by $[X] \mapsto \sum_k (-1)^k H^k_c(X, \QQ).$ The map extends to a map $\M_{\Lb} \to \M_{\Lb}^{Hdg}$ and a map $\hat{\M_{\Lb}} \to  \hat {\M_{\Lb}}^{Hdg}.$ 

There is a further specialisation of the map  $\hat{\M_{\Lb}} \to \hat{\M_{\Lb}}^{Hdg}$ to a map $$e : \hat{\M_{\Lb}} \to \ZZ [x,y] [[x^{-1}, y^{-1}]]$$ defined on generators by $$e(X) = \sum_{(p,q) \in \mathbb{N}^2}\sum_{k \in \mathbb{N}} (-1)^k h^{p,q} (H^k_c (X, \QQ)) x^py^q.$$ Here $h^{p,q}$  denotes the $(p,q)^{th}$ Hodge number of the mixed Hodge structure $H^k_c (X, \QQ)$. We will call the polynomial $e(X)$ the Serre polynomial of $X$.

\subsubsection*{Motivic zeta function}

Let $X$ be a quasi-projective variety over $\FF$. For $n >0$  we define $\Sym^n X$ to be the $n^{th}$ symmetric power of $X.$ For $n =0$, we set $\Sym^0 X$ to be $\Spec \FF$. 

We define the Kapranov motivic zeta function of $X$, $Z_X(t) \in \M_{\Lb} [[t]]$ by $$Z_X(t) = \sum_{k =0}^{\infty} [\Sym^k X] t^k.$$ We define $$\zeta_X (s) = Z_X (\Lb^{-s}).$$ For $s > \dim X$, $\zeta_x(s) \in \hat \M_{\LL}.$

\subsubsection{A note on brackets}
In this paper we will deal with many large formulae in the Grothendieck ring. As a result we have decided to omit the square brackets in our formulae when referring to the class of a variety, i.e. we will refer to $[X]$ as $X$. This is because keeping the brackets in the notation would lead to a huge number of brackets and be quite difficult to parse.
\subsection{Meta problems for convergence}
In this subsection we describe several meta problems about convergence of a sequence in the completion of a filtered ring. Theorems \ref{mainrate}, \ref{mainPn}, and \ref{maindim} will be instances of these metaproblems. This subsection is in a strict sense not essential to the rest of the paper and the main results of the paper will still be completely understandable without this subsection. However we include this section for motivation and to place some of our results in a wider context.

Let $R$ be a ring with an increasing  bi-infinite filtration $F^l$. Assume that $\cap_{\ell \in \ZZ} F^{\ell} R = 0, \cup_{\ell \in \ZZ} F^{\ell}R =R$. Let $\hat R$ be the completion of $R$ with respect to the filtration. 

Now we define a list of meta problems regarding convergence in $\hat R$ in increasing order of difficulty.
\begin{prob}
Let $((x_n))$ be a sequence in $\hat R$. Show that $\lim _{n \to \infty} x_n$ exists. 
\end{prob}

\begin{prob}
Let $((x_n))$ be a sequence in $\hat R$. Show that $\lim_{n \to \infty} x_n$ exists. Give an explicit description of the limit.
\end{prob}

\begin{prob}
Let $((x_n))$ be a sequence in $\hat R$. Show that $\lim_{n \to \infty} x_n = L$. 
Find $\ell_n \in \ZZ, \ell_n \to \infty$ such that $x_n - L \in F^{-\ell_n} 
\hat R$.
\end{prob}

\begin{prob}
Let $((x_n))$ be a sequence in $\hat R$. Show that $\lim_{n \to \infty} x_n = L$. Find $\ell_n \in \ZZ, \ell_n \to \infty$ such that $x_n - L \in F^{-\ell_n} 
\hat R$ and for $n \gg 0$ $x_n - L \not \in F^{-\ell_n +1} \hat R$.
\end{prob}

One could just as well propose a variant of Problem 4  where we  require $x_n - L \not \in F^{-\ell_n +C} \hat R $ where $C \in \RR$ is some constant not depending on $n$. 
\begin{prob}
Let $((x_n))$ be a sequence in $\hat R$ such that $\lim_{n \to \infty}x_n =L$ Construct a sequence $y_n$ such that $x_n -L = y_n + z_n$, where $z_n$ is much smaller than $y_n$.
\end{prob}

Of course Problem 5 as stated is quite imprecise and one would need to decide what it means for $z_n$ to be much smaller than $y_n$. For example, if the $y_n$ we construct are all units one could ask that, $\lim_{n \to \infty}\frac{z_n}{y_n} =0$.

\subsection{Statement of results on convergence}
We will adopt the following  notation:

\begin{enumerate}
\item Let $X$ be a scheme over a field $\FF$. For a closed subscheme $Y \subseteq X$ we denote the ideal sheaf of $Y$ by $I_Y$.
\item Let $\Fcal$ be a coherent sheaf on $X$. We define $h^i(X, \Fcal):= \dim _{\FF} H^i(X, \Fcal).$
\item We define $N_X^*$ to be the conormal sheaf of $Y$ in $X$ it is the pullback of the sheaf $I_Y$ to $Y$.
\item Given $X$ and $Y$ as above we define $N_X(Y)$ to be the subscheme of $X$ defined by the sheaf $I^2_Y$, i.e. the first infinitesmal neighbourhood of $Y$.

\end{enumerate}

Given $x,y \in \hat \M_{\Lb}$, we will say $x=y$ up to dimension $r$ if $x-y$ lies in the $F^{r}\hat \M_{\Lb}$, where $F^{r}$ refers to the dimension filtration on $\hat \M_{\Lb}$.

Let $X$ be a smooth variety over a field $\FF$, $\LL$ an ample  line bundle bundle on $X$. Then Proposition 3.7 of \cite{VW} states that 

$$ \lim_{d \to \infty} \frac{U(X, \LL^d)}{H^0(X, \LL^d)} = \zeta_X^{-1}(\dim X+1).$$
This can be seen as as an instance of Problem 1. However the proof of Proposition 3.7 in \cite{VW} actually gives more. Let $m(d)$ be the largest integer such that for all zero dimensional reduced subschemes $Z \subseteq X$ of length $ \le m(d)$, the map $H^0(X, \LL^d) \to H^0(N_X(Z) ,\LL^d)$ is surjective. Then the proof of Proposition 3.7 in \cite{VW} actually implies that  $$\frac{U(X, \LL^d)}{H^0(X, \LL^d)} - \zeta_X^{-1}(n+1)$$ lies in $ F^{-m(d)} \hat {\M_{\LL}}$. We improve this result to the following:
\begin{thm}\label{mainrate}
Let $X, \LL, m(d)$ be as above.  Then there exists $C> 0$ such that for all $d \gg 0$, $$\frac{U(X, \LL^d)}{H^0(X, \LL^d)} = \zeta_X^{-1}(n+1)$$ up to dimension  $-  m(d)(\dim X) + C$.
\end{thm}
\begin{rem}
  The constant $C$ in the above theorem depends on $X$ and $\LL$, it is related to the dimensions of the  spaces of curves in $X$ such that the line bundle $\LL$ restricts to a low degree. However, it is fairly straightforward to adapt the proof to obtain a constant $C'$, satisfying $C < C'$ where $C'$ depends only on the dimension of $X$.  
\end{rem}

We note that the improvement is greater for varieties of high dimension. This motivates us to study whether there are bounds on the difference 
$$\frac{U(X, \LL^d)}{H^0(X, \LL^d)} - \zeta_X^{-1}(\dim X+1)$$ that improve with the dimension of $X$ and are valid for small values of $d$. 

\begin{thm} \label{maindim}
Let $X$ be a smooth projective variety, let $\LL$ be a \emph{very} ample line bundle. Then for any $d \ge 3$,

$$ \frac{U(X, \LL^d)}{H^0(X, \LL^d)} = \zeta_X^{-1}(\dim X +1) ,$$ up to dimension $\sqrt{\dim X} /3$.
\end{thm}

As a corollary we have:
\begin{cor}\label{dimPn}
Let $X_n$ be a sequence of smooth projective varieties such that $\dim X_n \to \infty$ and $\lim_{n \to \infty}\zeta_X^{-1}(\dim X_n+1)$ exists. Let $\LL_n$ be a very ample line bundle on $X_n$. Let $d_n \ge 3$ be a sequence of positive integers. Then $$\lim_{n \to \infty} \frac{U(X_n, \LL_n^{d_n})}{H^0(X_n, \LL_n^{d_n})} = \lim_{n \to \infty}\zeta_{X_n}^{-1}(\dim X_n+1).$$

In particular $$\lim_{n \to \infty} \frac{U(\PP^n, \OO(d_n))}{H^0(\PP^n, \OO(d_n))} = \prod_{k=1}^{\infty}(1 - \Lb^{-k}).$$
\end{cor}

Finally we identify the leading error term for the sequence $\frac{U(\PP^n, \OO(d))}{H^0(\PP^n, \OO(d))}$, where $n$ is fixed and $d$  goes to infinity, providing a complete answer to Problem 5 in this instance.

\begin{thm} \label{mainPn}
Let $n \ge 2$.  Then for  $d \gg n$, there exists $\epsilon >0$. 
 $$\frac{U(\PP^n, \OO(d))}{H^0(\PP^n, \OO(d))} - Z_{\PP^n}^{-1}(n+1) = Y(d) + Z(d) $$ where: 
\begin{enumerate}
    \item $\dim Z(d) \le -(\frac{3n}{2} + \epsilon)d $.
    \item For $d \equiv 1 \mod{4}$ , the virtual Hodge structure of $Y(d)$ has highest weight term $E_{k }  (\Lb^{Hdg})^{n^2 - k(n+1) +1 }$ in $K_0(MHS),$ where $E_{d}$ is the highest weight term of the virtual Hodge structure corresponding to $H^1(\M_{1,1}, H_k)$. Here $\M_{1,1}$ is the moduli space of elliptic curves, $H_k$ is the Hodge bundle corresponding to $k^{th}$ tensor power of the relative canonical bundle and $k = \frac{3d+1}{2}$.
\end{enumerate}
 $Y(d)$ is explicitly defined in Section 6.
\end{thm}

\begin{rem}
    $Y(d)$ has an explicit description in terms of configurations of points on smooth plane cubic curves a precise description is given in Section 6 (defined just before Proposition \ref{Yd}) and its virtual Hodge structure may be readily computed from that description using the results of \cite{G}.
\end{rem}

\subsection{ Motivation for our results on convergence}

Our motivation for the results proven above comes from two different sources, homological stability and number theory. We will first discuss some motivation related to homological stability.

\subsubsection{Homological stability}
It was established by \cite{T} that the spaces $U(\PP^n ,\OO(d))$ satisfy homological stability. Later on Aumonier in \cite{A} and Das-Howe in \cite{DH} established homological stability for the sequence of spaces $U(X,\LL^d$. However, it was known (and this is even mentioned in \cite{T}) that the range of homological stability given in \cite{T} is not optimal. It was of interest therefore to see where and how homological stability fails and explain the first class in $H^*(U(\PP^n,\OO(d)), \QQ)$ that is not stable. 

Related to the above is the desire to establish `secondary homological stability'-like  results for the family of spaces $U(\PP^n, \OO(d)).$ For instance the paper \cite{GKRW} establishes that.

It would be interesting to establish some analogue of these results in the case of $U(\PP^n, \OO(d)),$ however there are several difficulties in trying to do so, first of all there are no maps between different $U(\PP^n, \OO(d))$ for different values of $d$ and secondly the nature of the spaces $U(\PP^n, \OO(d))$ is very different from the kind of spaces people have established `secondary stability' for; secondary stability results are known in the setting of configuration spaces on manifolds and other more topological and less algebraic geometric spaces as compared to discriminant complements. The techniques used to establish such results are much more topological than the ones in this paper.

In the end we were unable to establish such a `secondary stability' theorem for the homology of the spaces $U(\PP^n,\OO(d))$ and this paper does not really deal with results at the level of homology. 

One may in fact combine the techniques of this paper with those of \cite{T} to conclude that $H^k(U(\PP^n,\OO(d), \OO(d)) \cong H^k(GL_{n+1}(\CC), \QQ),$  for $k \le \frac{nd}{2}$. However we do not include this result as we do not believe this is optimal either and we do not how to obtain the optimal answer. In some sense our failure to get such a result is a consequence of some of the differences between cohomology and the Grothendieck ring, in the Grothendieck ring setting, if a variety $X$ admits a stratification with understandable pieces, we understand the class of $[X]$ in the Grothendieck ring. However if a variety $X$ has a stratification with understandable pieces, it is still not clear what the cohomology of $X$ is, the best we can do is obtain a spectral sequence whose $E_1$ page is related to the cohomology of the strata. However it is often not clear in these situations what the differentials in this spectral sequence are.

As a result we were unable to adapt some of the arguments involved in proving Theorem \ref{mainPn} to a cohomological setting. It still may very well be possible to establish a cohomological analogue of Theorem \ref{mainPn} but it would require somewhat different techniques than we have used/ are aware of.

However it is possible to establish that for fixed $ d \ge 3 $, $H^k(U(\PP^n ,\OO(d))) \cong H^k(GL_{n+1}(\CC))$ for $k \le f(n)$ where $f(n)$ is some function of $n$ that grows to $\infty$. This is a consequence of combining the techniques involved in proving Theorem \ref{maindim} with those of \cite{T}. //

We believe it should be possible to get a cohomological analogue of Theorem \ref{maindim} as well.

While we have not established any secondary stability results for the sequence of spaces $U(\PP^n, \OO(d))$, we have established an analogue of secondary stability in the setting of $K_0(Var_{\FF});$ by the results of \cite{VW}, we have $$ \lim_{d \to \infty} \frac{U(\PP^n, \OO(d))}{ H^0(\PP^n, \OO(d))} \to \frac{1}{Z_{\PP^n}(n+1)}$$ Theorem \ref{mainPn} refines this result to give us an estimate of the difference between the limit and the elements of the sequence. Furthermore the proof of Theorem \ref{mainPn} identifies this leading error term as arising from the failure of the proof of \cite{VW} to give an equality between  $\frac{1}{Z_{\PP^n}(n+1)}$ and $\frac{U(\PP^n, \OO(d))}{ H^0(\PP^n, \OO(d))}$; the proof of Theorem 1.13 in \cite{VW} crucially involves the fact that if $Z$ is a collection of points in $X$ whose length is small compared to $d$, then the vanishing of $1$-jets at the points of $Z$ impose linearly independent conditions. This fails when the length of $Z$ is comparable to $d$  and our error term $Y(d)$ is related to (and in fact defined in terms) of the variety parametrizing collections of points where this linear independence just begins to fail.

\subsubsection{Arithmetic statistics}
A second source of motivation comes from number theory, in particular from arithmetic statistics. In many situations in arithmetic statistics, one is interested in establishing asymptotic expressions for the number of some quantity of arithmetic interest. For instance as a prototype, one has Roberts conjecture (established in  \cite{BST}) which says that the number of degree $3$ extensions of $\QQ$ of discriminant $\le X$ is $aX + bX^{\frac{5}{6}} + o(X^{\frac{5}{6}})$. 

We view Theorem \ref{mainPn} as a Theorem of this kind, we have essentially established that the class of the discriminant complement $U(X,\LL)$ is $$H^0(\PP^n ,\OO(d))Z_{\PP^n}^{-1}(n +1) + H^0(\PP^n ,\OO(d)) Y(d) +\textrm{ lower order terms},$$ where lower order is to be understood in terms of the dimension filtration.

A natural question to ask here is if Theorem \ref{mainPn} and \ref{maindim} have consequences with regard to point counting over finite fields. More precisely, one may ask if one has a result of the following form:
Let $\FF = \FF_q$. Let $X ,\LL, d$ be as in Theorem \ref{maindim}.
Is there a bound on 
$$\frac{\# U(X ,\LL^d)}{\#3H^0(X ,\LL^d)} - \# Z_{X}^{-1}(-n -1)
$$ analogous to Theorem \ref{mainPn}?
Does  $$\lim_{n \\to \infty} \# \frac{U(\PP^n ,\OO(d))}{H^0(\PP^n ,\OO(d))}$$ exist and equal $$\prod_{i =1}^{\infty}(1 - \frac{1}{q^i})?$$

Does one have $$\# \frac{(U(\PP^n ,\OO(d)))}{H^0(\PP^n, \OO(d))} = \# Z_{\PP^n}(-n-1) +\# Y(d) + \textrm{ lower order terms}?$$

We believe that the answer to all the above questions is yes and we are currently working on solving the above questions. The reason why point counting formulae like the above do not immediately follow from Theorems \ref{maindim} and \ref{mainPn} is that the functor $\#$ is not continuous with respect to the dimension filtration, varieties of small dimension can have an arbitrary number of points. In addition to this some of the methods we use are not well adapted to the setting of point counts, e.g. the sum over ordered partitions constructions that we use involve a very large number of terms that may cause problems when added up. However we believe that we may combine our methods with that of \cite{Po} to obtain the above results.
\subsection{Cayley-Bacharach sets}
In this section, we will define what Cayley-Bacharach sets are and prove a few basic propositions about them. This may appear to be quite unrelated to the material of the previous section, however it is not as the results of the previous section rely on our results on Cayley-Bacharach sets.

In this section, we assume $\FF$ to be an algebraically closed field. Let $X$ be a projective variety over $\FF$. Let $\LL$ be a line bundle on $X$.

\begin{defn} 
A  zero dimensional reduced subscheme $Z  \subseteq X$ is said to be Cayley-Bacharach for $\LL$ if given $f \in H^0(X, \LL)$, if $f$ vanishes on all but one point of $Z$ then $f$ vanishes on all of $Z$.
\end{defn}

We will now reformulate the above definition. Given any $Z' \subseteq Z \subseteq X$ we have  restriction maps $H^0(X, \LL) \to H^0(Z, \LL) \to H^0(Z',\LL)$. We hence have a map $$\textrm{Im} (H^0(X, \LL) \to H^0(Z, \LL))\to \textrm{Im}(H^0(X, \LL) \to H^0(Z', \LL)).$$

Given a zero dimensional scheme $Z$, we will use $|Z|$ to denote the length of $Z$.

\begin{prop}
    A zero dimensional reduced subscheme $Z \subseteq X$ is  Cayley-Bacharach for $\LL $ iff for all $Z' \subseteq Z$, $|Z \setminus Z'| =1$, the restriction map 
    \begin{equation}   
        \textrm{Im} (H^0(X, \LL) \to H^0(Z, \LL))\to \textrm{Im}(H^0(X, \LL) \to H^0(Z', \LL))
    \end{equation}
 is an isomorphism. 
\end{prop}
\begin{proof}

    If $Z' \subseteq Z$, then the homomorphism in (1) is always surjective. For it to be to an isomorphism is equivalent therefore to it having zero kernel. But if $f$ is in the kernel of (1), it is an element of $H^0(X, \LL)$  that restricts to $0$ on $Z'$. If $Z$ is Cayley-Bacharach  this implies that $f$ rstricts to $0$ on $Z$ and hence the kernel of the (1) is trivial. Conversely if the kernel of (1)  is trivial, then any $f$ restricting to zero on $Z'$ restricts to $0$ on $Z$, and hence if the kernel of (1) is trivial for all choices of $Z'$, then $Z$ is Cayley-Bacharach for $\LL$.
\end{proof}

\subsection{Results on Cayley-Bacharach sets}
In\cite{Pi}, Picoco asks the following question.

\begin{qn}
Let $e,d >0$. Suppose $Z \subseteq \PP^n$ is a reduced set of points and that $|Z| < e(d-e +3) -1$.
Then is it true that if $Z$ is Cayley-Bacharach for $\OO(d)$, it lies on a curve of degree $< e$?
\end{qn}

While we do not know if Picoco's question has an affirmative answer we have made the following partial result.
\begin{thm} \label{mainCB}
Let $d \gg e \ge 1$.  Let $Z$ be a Cayley Bacharach subset of $\PP^n$ for $\OO(d)$. Then there exists a function $f$, depending  on $n$  but not on $d$ such that if $$|Z| < ed - f(e),$$ $Z$ lies on a curve of degree $\le e$.
\end{thm}
In \cite{LU}(Conjecture 1.2 in the paper) Levinson and Ullery made the following conjecture 
\begin{conj}\label{conjL1}
Let $d \ge 1$. Suppose $Z \subseteq \PP^n$ is a Cayley-Bacharach set for $\OO(d)$. Suppose $|Z| \le (e+1)d +1$ Then $Z$ is contained in a union of positive dimensional linear subspaces $L_1, \dots L_k$ of $\PP^n$ such that $\sum_{i=1} ^{k}\dim L_i \le e$.
\end{conj}

\begin{thm}\label{conjLU}
Let $d \gg e$. Then Conjecture \ref{conjL1} is true.
\end{thm}
One can replace $d \gg e$ in \ref{conjLU} with $d$ greater than a certain quartic polynomial in $e$, by going through the proof. 
The proof of Theorem \ref{conjLU} (assuming Theorem \ref{mainCB} ) is fairly straightforward and is essentially a consequence of the fact that  an integral curve of degree $e$  lies in an $e$ dimensional linear subspace of projective space.

\subsection{Acknowledgements}
I would like to thank Eduard Looijenga and Madhav Nori for helpful and clarifying discussions. I'd like to thank Jake Levinson, Brooke Ullery and David Stapleton for discussing some of their ideas with me. I'd like to thank Ronno Das, and Aaron Landesmann for suggestion on how to improve the paper. I'd like to thank Jordan Ellenberg for a question that led to Theorem \ref{maindim}. Finally, I would like to thank Joshua Mundinger for extensive comments on several drafts of this paper.
\section{Cayley-Bacharach subsets lie on curves}
In this section we will prove Theorems \ref{mainCB} and \ref{conjLU} . Our strategy to prove Theorem \ref{mainCB} is to induct on the dimension of the ambient space. In the case of $\PP^2$ (the base case of our induction), we have the following theorem of Lopez and Pirola (this is a subcase of Lemma 2.5 of \cite{LP}).
\begin{prop} \label{proplp}
    Let $Z \subseteq \PP^2$ be a Cayley Bacharach set for $\OO(d)$. Let $e \ge 1$. If $|Z| \le e(d- e + 3) - 1$, then $Z$ lies on a curve of degree $\le e-1$.
\end{prop}

Our induction strategy will involve linearly projecting onto smaller dimensional projective spaces. We will need the following propositions in our proof, which deal with projections and Cayley Bacharach sets.

\begin{prop}\label{projCB}
Let $n \ge 2$. Let $Z \subseteq \PP^n$ be a Cayley-Bacharach set for $\OO(d)$. Let $P \in \PP^n \setminus Z$. Let $H \cong \PP^{n-1}$ be a hyperplane not containing $P$. Let $\pi: \PP^n \setminus \{ P\} \to H$ be the projection map. Suppose that the scheme theoretic image of $Z$, $\pi(Z)$  is reduced. Then $\pi(Z)$ is a Cayley-Bacharach set for $\OO(d)$.
\end{prop}
\begin{proof}
 We note that given a linear projection $\pi$, defined by projecting from a point $P$ and a zero dimensional reduced subscheme $Z$, $\pi(Z)$; the conditions that $P$ does not lie on a line joining any two of the points in $Z$ and that $\pi(Z)$ is reduced are equivalent.
 
Without loss of generality, we may assume $P =[0:0: \dots :0: 1]$ and $H$ is the hyperplane defined by $ X_{n+1} =0.$ Then the projection map $\pi$ is defined by $[x_1: x_2 : \dots: x_n: x_{n+1}] \mapsto [x_1: x_2 : \dots: x_n].$ Let $Z = \{P_1, P_2,\dots, P_{m}\}.$
Let $\pi(Z) = \{\pi(P_1),\pi(P_2), \dots, \pi(P_m)\}$. Let $f \in H^0(\PP^{n-1}, \OO(d))$ be a section vanishing at $\{\pi(P_1), \dots, \pi(P_{m-1})\}. $ Such an $f$ is given by a degree $d$ homogenous polynomial $p$ in the variables  $X_1 \dots, X_n$. However the $p$ can be thought of as a polynomial in the variables $X_1 \dots, X_n, X_{n+1}$ as well, with no dependence on $X_{n+1}$. We observe that for any $Q \in \PP^{n} - \{[0: \dots: 0:1]\}$, $p(Q ) =0$ iff $p(\pi(Q))=0$. Hence $p(P_i) =0$ for $i \in \{1, \dots m-1\}.$ But since $Z$ is Cayley-Bacharach for $\OO(d)$ this implies $p$ vanishes on $P_m$ as well. But this implies $p$ vanishes on $\pi(P_m)$. Hence $\pi(Z)$ is Cayley-Bacharach for $\OO(d)$.
\end{proof}
\begin{prop}\label{intersectioncone}
Let $n \ge 3$. Let $P_1, P_2$ be two distinct points in $  \PP^n$. Let $H$ be a hyperplane in $\PP^n$ such that $P_1, P_2 \not \in H$. Let $C_1, C_2$ be two curves in $H$ such that $\textrm{deg}(C_i) \le e_i$. Let $S_i$ be the cone over $C_i$ with apex $P_i$. Then $S_1 \cap S_2 =  C \cup \Gamma$, where $C$ is a curve and $\Gamma$ is zero dimensional. Furthermore $\deg(C) + |\Gamma| \le e_1e_2$.
\end{prop}
\begin{proof}
We begin by noting that a cone over an irreducible curve is irreducible and that cones with  different apexes are distinct subvarieties.  This immediately implies that $S_1$ and $S_2$ have no irreducible components in common and $S_1 \cap S_2$ is a lower dimensional subvariety of $\PP^n$ and hence is of the form $C \cup \Gamma$, where $C$ is a curve and $\Gamma$ is zero dimensional.

We note that $\deg(S_i) = \deg(C_i)$. We now apply the generalised Bezout's theorem of \cite{F} (bottom of page 223), which states that the total degree of $S_1 \cap S_2$ (i.e. $\deg C + |\Gamma|$) is less than or equal to $\deg(S_1) \deg(S_2) \le e_1 e_2$.
\end{proof}

\begin{prop} \label{curveonly}
Let $d \gg e >0$. Let $Z \subseteq \PP^n$ be a Cayley-Bacharach set for $\OO(d)$. Let $C \subseteq \PP^n $ be a reduced curve. Let $\Gamma \subseteq \PP^n$ be a finite set of points disjoint from $C$. Suppose $\deg(C) + |\Gamma| \le e^2$. If $Z \subseteq C \cup \Gamma$,  then $Z \subseteq C$.
\end{prop}
\begin{proof} 
For a fixed value of $\deg(C) + |\Gamma|$, there exists $d_0$ such that for $d >d_0$, the map $H^0(\PP^n, \OO(d)) \to H^0(C \cup \Gamma, \OO(d)) $ is surjective.  Suppose $P \in Z \cap \Gamma$, and that $P \not \in C$. Then there exists a section $f \in H^0(C \cup \Gamma, \OO(d))$ that is nonzero at $P$ and vanishes on $C \cup \Gamma \setminus \{P\}$. Since the map $H^0(\PP^n, \OO(d)) \to H^0(C \cup \Gamma, \OO(d)) $ is surjective, we can find a polynomial  $p \in H^0(\PP^n, \OO(d))$ that restricts to $f$. Such a $p$ vanishes on $Z \setminus\{P\}$ and doesn't vanish on $P$. This contradicts the Cayley-Bacharach property of $Z$. Hence no such $p$ exists, i.e. $Z$ is contained in $C$.
\end{proof}

\begin{prop}\label{curveCB}
Let $d \gg e >0$. Let $C \subseteq \PP^n$ be a reduced curve of degree $e$. Let $Z \subseteq C$ be a Cayley-Bacharach set for $\OO(d)$. We assume that $Z$ is not contained in any curve of degree $< e$. Then there is a positive real valued increasing function $f$ not depending on $Z$ or $d$ such that $|Z| \ge ed - f(e)$.
\end{prop}
\begin{proof}
    The proof of this proposition is essentially Riemann-Roch for curves. However, we first need to do some work to convert the given curve $C$ into a smooth curve to which we can apply the Riemann Roch theorem.

    We first note that a Cayley-Bacharach set must contain at least $d$ points (see for instance \cite{LU}). Furthermore there is a function $\phi_1$, such that any curve $C_0$ of degree $e$ satisfies  $ |\Sing(C_0)|\le \phi_1(e)$ . Hence for $d \gg e$, $Z$ is not contained in $\textrm{Sing}(C)$. Let $C_1, \dots, C_k$ be the irreducible components of $C$. We note that $k \le e$.

    We can similarly conclude that for any irreducible component $C_i$ of $C$, the intersection $Z \cap (C_i \setminus \textrm{Sing}(C)) $ is nonempty. If this were not the case, $Z$ would be a Cayley-Bacharach set contained in the union of a curve $C' = \cup_{j \neq i} C_j$ and a finite set, $\Sing C_i$. But for $d \gg e$,we may apply the argument of Proposition \ref{curveonly} to conclude  $Z \subseteq C'$ which contradicts the fact that $Z$ is not contained in a curve of degree $<e$.

    Let $\pi:\tilde C \to C$ be a resolution of singularities. Let $\tilde Z = \pi^{-1}(Z) \cup \pi^{-1}(\textrm{Sing} C)$.  Note that there is some function $\phi_2$ such that  $|\tilde Z| - |Z| < \phi_2(e)$. Let $P \in Z \setminus \textrm{Sing}(C)$. Let $\tilde P = \pi^{-1}(P)$. Suppose there exists $ \tilde f \in H^0(\tilde C , \OO(d))$ vanishing at $\tilde Z - \tilde P$, such that $ \tilde f(\tilde P) \neq 0$. Then $\tilde f$ is the pullback of a section $f \in H^0(C, \OO(d))$ vanishing at $Z \setminus\{P\}$ and not vanishing at $P$. For $d\gg e$, $H^0(\PP^n, \OO(d))$ surjects onto $H^0(C, \OO(d))$ and hence $f$ is the restriction of some element of $H^0(\PP^n, \OO(d))$. This contradicts the Cayley-Bacharach property of $Z$. Hence no  $f$ and consequently no such  $\tilde f$ can exist. We may now  apply the Riemann Roch theorem to the curve $\tilde C$. 

     Let $\tilde C_1, \dots, \tilde C_k$ be the corresponding components of $\tilde C$.

    Let $ \tilde P \in \pi^{-1}( (Z \cap C_i) \setminus \Sing(C))$. Let $\pi(\tilde P) = P$. We note that any $\tilde f \in H^0(\tilde C_i, \OO(d))$ that vanishes on $\pi^{-1}(\Sing(C)) \cup  \pi^{-1}(Z) \setminus \{P\}$ is automatically the pullback of a section  $f \in H^0(C, \OO(d))$ vanishing at $Z \setminus\{P\}$ and hence $\tilde f$ is forced to vanish at $P$. Let $Y_i = \pi^{-1}(\Sing(C)) \cup  \pi^{-1}(Z) $. We have established that the restriction map  $H^0(\tilde C_i, \OO(d)) \to H^0(Y_i, \OO(d)) $ is not surjective. Let $|Y_i|$ denote the line bundle on $\tilde C_i$ associated with the divisor $Y_i$. Thus,   $$H^1(\tilde C_i, \OO(d) \otimes |Y_i|^{\vee}) \neq 0.$$ By Riemann-Roch this implies that \begin{equation}\label{ineq1}
        d \deg C_i - |\pi^{-1}(Z \cap C_i)| - |\pi^{-1}(\Sing(C) \cap C_i)| \le \textrm{genus(C)}.
    \end{equation}

      If we sum (\ref{ineq1}) over all $i$ we obtain $ d \deg(C) -  \sum_i(|\pi^{-1}(Z \cap C_i)| - |\pi^{-1}(\Sing(C) \cap C_i)|)\le \sum_i \textrm{genus}(C_i) .$ Hence  
      \begin{equation}\label{ineq2}
          \sum_i(|\pi^{-1}(Z \cap C_i)| \ge d \deg(C) - \sum_i \textrm{genus}(C_i)  - \sum |\pi^{-1}(\Sing(C) \cap C_i)|).
      \end{equation}

  We  note that $|\pi^{-1}(Z \cap C_i)| - |Z \cap C_i| \le \phi_2(e)$, with $\phi_2$ as defined above. 
      Hence, $|Z \cap C_i| \ge d(\deg C_i) - \phi_2(e) - \phi_1(e)$.
      Furthermore $\sum_i |Z \cap C_i| -|Z| \le k \Sing(C) \le e \phi_1(e).$

   Combining this with \ref{ineq2} we have
   $$|Z| \ge d \deg(C) - \sum_i(\textrm{genus}(C_i)) -  \sum_i |\pi^{-1}(\Sing(C) \cap C_i)|) - e \phi_2(e)$$ $$ \ge  d \deg(C) - \sum_i(\textrm{genus}(C_i)) - 2 e \phi_2(e) .$$

   Now we note that the genus of each individual $C_i$ is bounded in terms of $e$ and $n$ by the Castelnuovo bound. Let us call this bound $\phi_3(e)$. This gives us 

   $$|Z| \ge d\deg(C) - e(\phi_3(e)+ 2 \phi_2(e)).$$ We let $f(e)= \max_{j \le e}(1, j(\phi_3(j)+ 2 \phi_2(j)) )$. This establishes the result.
\end{proof}

We now prove Theorem \ref{mainCB}.

\begin{proof}[Proof of Theorem \ref{mainCB}]
    We will prove the theorem by induction on $n$. For $n =2$, the Theorem follows from Proposition \ref{proplp}.  Let us assume the Theorem for Cayley-Bacharach sets in  $\PP^{n}$ for some $n\ge 2$. Let $Z \subseteq{\PP^{n+1}}$ be a Cayley Bacharach set for $\OO(d)$, with $|Z| < ed -f(e)$. Let $P_1, P_2$ be two distinct points disjoint from $Z$ such that no line passing through two of the points of $Z$ contains $P_1$ or $P_2$. Let $H$ be a hyperplane in $\PP^{n+1}$ not containing $P_1$ or $P_2$. We project $Z$ linearly from $P_i$  onto $H$ to obtain a  set $Z_i$. The set $Z_i$ is Cayley-Bacharach by Proposition \ref{projCB}. By the induction hypothesis $Z_i$ lies on a curve $C_i$ of degree $e_i \le e$. Therefore $Z$ lies on the intersection of  $S_1$ and $S_2$, where $S_i$ is the cone over $C_i$ with apex $P_i$. By  Proposition \ref{intersectioncone} $S_1 \cap S_2 = C \cup \Gamma$ where $\deg(C) + |\Gamma| \le e_1 e_2$.  By  Proposition \ref{curveonly}, $ Z \subseteq C$.  Let us replace $C$ by the curve of minimal degree containing $Z$.  By Proposition \ref{curveCB}, $|Z| \ge \deg(C) d - f(\deg (C))$.  But our assumption implies that $|Z| \le  e d - f(e)$.
    Since $d \gg e$, this implies $\deg(C) \le e$.
    
\end{proof}

We now move on to the proof of Theorem \ref{conjLU}.

\begin{proof} [Proof of Theorem \ref{conjLU}]
    Let $Z$ be the given Cayley-Bacharach set. By Theorem \ref{mainCB}, for $d \gg e$ and  $(e+1)d +1 \le ed - f(e)$ $ Z$ is contained in a curve $C$ of degree at most $e+1$.  We remind the reader that  an integral curve of degree $e$ is contained in an $e$ dimensional linear subspace and an integral curve of degree $e$ that is not contained in an $e-1$ dimensional subspace is a rational normal curve.
   
    As a result the  only curves of degree $e +1$ that are not contained in a union of linear subspaces whose dimensions sum to  $\le e$ are unions of rational normal curves. 
     Hence for our purposes we may assume $Z$ is a Cayley- Bacharach set for $\OO(d)$ that is contained in a union of $k$ rational normal curves $C_i$ with $\deg(C_i) = r_i$ and $\sum r_i = e+1$, as the result follows in all other cases.

     .
    Suppose two rational normal curves $C_i$ and $C_j$ intersect at $k$ points. Then by the properties of rational normal curves one may conclude that their union $C_i \cup C_j$ is contained in a subspace of dimension $\deg C_i + \deg C_j + 1 -k$. So if $k>1$ for any pair $(i,j)$ we may replace $P_i, P_j$ by the span of $P_i$ and $P_j$ to obtain a new sequence of linear subspaces whose dimensions sum to $\le e$. 

Hence we may assume $|C_i \cap C_j| \le 1$ as the result follows in all other cases.

Furthermore we claim that the intersection graph of the $C_i$ is a forest. This is established as follows:

If $$|C_i \cap C_j| = |C_j\cap C_k| = |C_i \cap C_k| = 1$$, then $\dim \textrm{span} P_i \cup P_j \cup P_k < \dim P_i + \dim P_j + \dim P_k$ unless $|C_i \cap C_j \cap C_k| \neq 0$.
A similar argument implies that there are no  cycles in the intersection graph, and hence it is a forest.

Let $d(C_i)$ denote the degree of the vertex $C_i$ in the intersection graph. We claim that  $|Z \cap C_i| + d(C_i) \ge r_i d +2$. For suppose that $|Z \cap C_i| + d(C_i) < r_i d +2$.  Let $S$ be the set of points in $C_i$ that are also in some $C_j$ for $j \neq i$. Let $C_i^{\circ} = C_i \setminus S$, i.e. the interior of $C_i$ in $C$. Let $p \in Z \cap C_i^{\circ} $. We note that $|Z \cap C_i^{\circ} \setminus{p}| + |S| \le |Z\cap C_i| - 1 + d(C_i) \le r_i d$, by assumption. However, given any set of $ \le r_id$ points not containing $p$ in $C_i$, there is a section of $f \in H^0(C_i, \OO(d))$ vanishing at those points and not at $p$ (Here we use the fact that $C_i$ is isomorphic to 
$\PP^1$). We may extend $f$ from $C_i$ to $C$, by defining it to be zero outside $C_i$. Then $f$ vanishes on $Z \setminus \{p\}$ but not at $p$. This contradicts the Cayley-Bacharach property of $Z$. Hence 
  $|Z \cap C_i^{\circ}| + d(C_i) \ge r_i d +2$. 

  By summing this inequality over $i$, we obtain $$|Z| + \sum_i(d(C_i)) \ge \sum_i |Z \cap C_i^{\circ}| +  \sum_i d(C_i) \ge  \sum_i r_id + 2k = (e+1)d + 2k.$$  By the above inequality, to establish that $|Z| \ge d(e+1)$ suffices to establish that $2k - \sum_i d(C_i) \ge 1$. We note that 
$\sum_i d(C_i)$ is twice the number of edges in the intersection graph. As a result, $2k - \sum_i d(C_i)$ is twice the Euler characteristic of the intersection graph. Since the intersection graph is a forest, the Euler characteristic is always $\ge 1$.

\end{proof}
\section{Preparatory lemmas}
The purpose of this section is to state and prove Lemma \ref{contcurve}  which will be used later on in the paper to prove our theorems on convergence of the motives of discriminant complements. It is essentially a corollary of Theorem \ref{mainCB} .

\begin{lem}\label{contcurve}
Let $d \gg e >0$. Let $f$ be as in Theorem \ref{mainCB} Let $Z \subseteq \PP^n$ be a reduced zero dimensional subscheme.  Suppose $2|Z| \le ed - f(e) $.  Suppose $h^1(\PP^n ,I^2_Z(d)) \neq 0$. Then there is a function $f$  such that if $2|Z| \le ed - f(e)$, then there is a subset $Z' \subseteq Z$ and a curve $C$ of degree $\le e$ such that $h^1(\PP^n ,I^2_Z(d)) = h^1(\PP^n, I^2_{Z'}(d)) =h^1(N_{\PP^n}(C), I^2_{Z'}(d))$. Furthermore the curve of minimal degree satisfying the above conditions is unique.

\end{lem}
Before embarking on the proof of lemma, we will need  a few propositions.
\begin{prop}\label{contCB}
    Assume that  $\FF$  is algebraically closed. Let $Z \subseteq \PP^n$ be a reduced zero dimensional subscheme. Suppose that the restriction map   $H^0(\PP^n, \OO(d)) \to H^0(Z, \OO(d))$ is not surjective.  Then $Z$ has a subset $Z'$ that is Cayley-Bacharach for $\OO (d)$.
\end{prop}
\begin{proof}
Let $Z' \subseteq Z$ be such that $$H^0(\PP^n, \OO(d)) \to H^0(Z', \OO(d))$$ is not surjective and for any proper subset $Z'' \subseteq Z'$, the restriction map  $$H^0(\PP^n, \OO(d)) \to H^0(Z'', \OO(d))$$ is surjective. Clearly, such a $Z'$ exists as the restriction map is surjective for singleton sets. Let us now take any subset $Z'' \subseteq Z'$ such that $|Z' \setminus Z''| =1$. 

Since $H^0(\PP^n, \OO(d)) \to H^0(Z', \OO(d))$ is  not surjective,  $$ \dim Im(H^0(\PP^n, \OO(d)) \to H^0(Z', \OO(d))) < \dim H^0(Z', \OO(d))) = |Z'|.$$
However $Im(H^0(\PP^n, \OO(d)) \to H^0(Z', \OO(d)))$ surjects onto $$Im(H^0(\PP^n, \OO(d)) \to H^0(Z'', \OO(d))) = H^0(Z'', \OO(d))$$ which is of dimension $|Z''| =|Z'| -1$. However, since the dimension of the source  is at most $|Z'| -1$, the restrcition map is forced to be an isomorphism, which establishes that $Z'$
is Cayley Bacharach.

\end{proof}

\begin{prop}\label{d/2}
Assume that $\FF$  is algebraically closed. 
Let $Z \subseteq \PP^n$ be a reduced zero dimensional subscheme. Suppose   $h^1(\PP^n ,I^2_Z(d)) \neq 0$. Then $Z$ has a subset $Z'$ that is Cayley-Bacharach for $\OO (\lceil(d/2)\rceil)$.

\end{prop}
\begin{proof}
    We observe that Theorem 1.1 of \cite{GGP} of implies that if $h^1(\PP^n ,I^2_Z(d)) \neq 0$, then $h^1(\PP^n ,I_Z(\lceil(d/2)\rceil) \neq 0$. Associate to the short exact sequence  $$0 \to I_Z(\lceil(d/2)\rceil) \to \OO(\lceil(d/2)\rceil) \to \OO_Z(\lceil(d/2)\rceil)\to 0$$, i we have a long exact sequence in cohomology, part of which is as follows:
   $$ H^0(\PP^n, \OO(\lceil(d/2)\rceil)) \to H^0(Z, \OO(\lceil(d/2)\rceil)) \to H^1(\PP^n, I_Z(\lceil(d/2)\rceil)) \to H^1(\PP^n, \OO(\lceil(d/2)\rceil)))$$ 
   But $H^1(\PP^n, \OO(\lceil(d/2)\rceil))) \neq 0,$.
   As a result the fact that  $h^1(\PP^n ,I_Z(\lceil(d/2)\rceil) \neq 0$implies that the map $H^0(\PP^n \OO(\lceil(d/2)\rceil)) \to H^0(Z \OO(\lceil(d/2)\rceil))$ can not be surjective.
    
      We then apply Proposition \ref{contCB} to conclude that $Z$ has a subset $Z'$ that is Cayley-Bacharach for $\OO (\lceil(d/2)\rceil)$.
\end{proof}
\begin{prop}\label{CBdec}
Let $e, n>0$. Let $Z \subseteq \PP^n$ be a reduced set of points.  Let $C$ be a curve of degree $e$. Let $d \gg e,n$. Suppose   $h^1(\PP^n ,I^2_Z(d)) \neq 0$. Let $Z' = Z \cap C$. Suppose $$h^1(\PP^n, I^2_Z(d)) > h^1(\PP^n, I^2_{Z'}(d)).$$ Let $Z'' = Z \setminus Z'$. Then $H^0(\PP^n, \OO(d-2e)) \to H^0(N(Z''), \OO(d-2e))$ is not surjective.
\end{prop}
\begin{proof}

    Let $D \subseteq \PP^n$ be a divisor containing $C$  not containing any point of $Z''$. Let $|D|$ denote the degree of $D$.  
    Suppose for the sake of contradiction that $H^0(\PP^n, \OO(d - 2|D|)$ surjects onto $H^0(N(Z''), \OO(d-2|D|)$. Let $f \in H^0(\PP^n, |D|)$ be a section whose zero locus is $D$.  We have a subspace $$f^2 \otimes H^0(\PP^n, \OO(d-2|D|) \subseteq H^0(\PP^n, \OO(d))$$ that surjects onto $H^0(Z'', \OO(d))$ and maps to $0$ in $H^0(N(Z'), \OO(d))$. We note that $H^0(N(Z), \OO(d)) \cong H^0(N(Z'), \OO(d)) \oplus H^0(N(Z''), \OO(d)).$ Some elementary linear algebra then implies that the codimension of the image of $H^0(\PP^n, \OO(d))$ in $H^0(N(Z), \OO(d))$ is equal to the codimension of the image of $H^0(\PP^n, \OO(d))$ in $H^0(N(Z'), \OO(d))$, since the map surjects onto the other factor. This contradicts the assumption that $$h^1(\PP^n, I^2_Z(d)) > h^1(\PP^n, I^2_{Z'}(d)).$$
\end{proof}

We now proceed to prove Lemma \ref{contcurve}

\begin{proof}[Proof of Lemma \ref{contcurve}]
Let us first prove the lemma under the assumption that our base field $\FF$ is algebraically closed.
We begin by applying Proposition \ref{contCB} to conclude that $Z$ contains a subset $Z_0$ that is Cayley Bacharach for $\OO(\lceil d/2 \rceil)$. We then apply Theorem \ref{mainCB} to conclude that there is a curve $C$ that contains $Z_0$. Let $Z' = Z \cap C$. If $h^1(\PP^n , I^2_Z(d)) = h^1(\PP^n, I^2_{Z'}(d)) $ we are done and $C$ is our required curve. If not, we proceed as follows:
\begin{enumerate}
    \item By Proposition \ref{CBdec} $Z'' = Z \setminus Z'$ must satisfy $h^1(\PP^n, I^2_{Z''}(d-2e) >0$. Hence it must contain a Cayley Bacharach subset $Z'''$ for $\OO(\lceil d/2\rceil-e)$, which is contained on a curve $C'$.
    \item We replace our original curve $C$ by $C \cup C'$ and replace $Z'$ by $Z \cap (C \cup C')$. 
    \item We repeat steps 1 and 2 until $h^1(\PP^n , I^2_Z(d)) = h^1(\PP^n, I^2_{Z'}(d)) $. Since $|Z'|$ increases every time we perform step $2$ and  $|Z'| < |Z|$ this process terminates.
\end{enumerate}
We then have our required curve $C$ and subset $Z'$.

If our base field $\FF $ is not algebraically closed, we may pass to  the algebraic closure $\bar \FF$ and then find a curve $C$ defined over $\bar\FF$ containing $Z$ and of degree $\le e$. We will prove that $C$ is defined over $\FF$ which will complete the proof of the lemma. Let $g \in Gal(\bar \FF, \FF)$. We claim that $g(C) = C$. If $g(C) \neq C$, there is some irreducible component $C_i \subseteq C$ such that $g(C_i) \neq C_i$. But $ Z \cap C_i \subseteq |g(C_i) \cap C_i | \le \deg (C_i)^2 \le e^2.$ However, $Z \cap C_i \ge e^2 $ (see proof of Proposition \ref{curveonly}). Hence for $d \gg e$, $g(C_i) = C_i$ and hence $C$ is defined over $\FF$. This completes the proof of existence.

Let us now turn to proving uniqueness. It suffices to prove uniqueness over an algebraically closed field. Suppose that $C_1, C_2$ are curves of the same degree, with $\deg (C_i) \le e \ll d$, such that $h^1(\PP^n ,I^2_Z(d)) = h^1(\PP^n, I^2_{Z_i}(d)) = l$ where $Z_i = Z \cap C_i$ and no curve of strictly smaller degree satisfies the conditions of the lemma. Let $W = Z_1 \cap Z_2$.  Suppose $h^1(\PP^n, I^2_{W}(d)) < h^1(\PP^n, I^2_Z(d)) =l$. We will show that this implies that $h^1(\PP^n, I^2_Z(d)) >h^1(\PP^n, I^2_{Z_i}(d))$ which contradicts our assumption. Let $Z_i' = Z_i \setminus W$. Let $U \subseteq H^0(Z_1' \cup Z_2' \cup W, \OO(d))$ be the image of $H^0(\PP^n ,\OO(d))$ which is of codimension $l$. Let $\pi_i: H^0(Z, \OO(d)) \to H^0(Z_i, \OO(d))$ be the restriction map. Let $\pi: H^0(Z, \OO(d)) \to H^0(W, \OO(d))$ be the restriction map. By our assumption, $ \pi_i(U)$ is of codimension $l$ in $H^0(Z_i, \OO(d))$. Therefore $U \subseteq \pi_i^{-1}(\pi_i)(U)$ and both are of codimension $l$ in $H^0(Z, \OO(d))$. Hence we have $$U = \pi_1^{-1}(\pi_1(U)) = \pi_2^{-1}(\pi_2(U)).$$

The only way the above equation can hold is if $U = \pi^{-1}(\pi(U))$ and as a result

$\pi(U)$ is also of codimension $l$, which contradicts $h^1(\PP^n, I^2_{Z_0}(d)) < h^1(\PP^n, I^2_Z(d)) =l$.

Hence $h^1(\PP^n, I^2_{Z_0}(d)) = h^1(\PP^n, I^2_Z(d)) =l > 0$. However in this case we note that $Z_0  \subseteq C_1 \cap C_2$. Let $C_1 \cap C_2 = C \cup \Gamma$ where $C$ is a reduced curve and $\Gamma$ is a finite set of points. Either $h^1(\PP^n, I^2_{Z \cap C}(d)) = h^1(\PP^n, I^2_Z(d))$, in which case $C$ is a curve of smaller degree than $C_1$ or $C_2$ satisfying the conditions of the lemma. Otherwise $h^1(\PP^n, I^2_{Z \cap C}(d)) < h^1(\PP^n, I^2_{Z_0}(d))$ However, this implies that $Z_0 \setminus (Z \cap C) \subseteq \Gamma$ contains a Cayley Bacharach set for $\OO(d - 2 \deg C)$ by the combination of Proposition \ref{CBdec} and Proposition \ref{d/2}. But $|\Gamma | \le e^2$ and a Cayley Bacharach set for $\OO(d - 2 \deg C)$ must contain at least $d-2e $  points. Since $d \gg e$, we are done.

\end{proof}
\section{Proof of Theorem \ref{mainrate}}

In this section we prove Theorem \ref{mainrate}. 

We will begin by reviewing some material in section 2 of \cite{VW}.

\begin{defn}
    Let $S$ be a nonempty finite set. An ordered partition $\lambda$ of the set $S$ is a tuple $(A_1, \dots ,A_m)$, where the $A_i$ are nonempty disjoint subsets of $S$ such that $A_1 \cup  \dots \cup A_m =S$. 
    The empty set $\varnothing$ has a single ordered partition called the empty partition and  also denoted $\varnothing$.  We will denote the ordered partition of $\{1, \dots, n\}$ with only one set by $[n]$.
\end{defn}

We can associate two numbers to an ordered partition $\lambda$. If $\lambda = (A_1, A_2, \dots. A_m)$ is an ordered partition of a nonempty finite set $S $, we define $|\lambda| =|S|$ and $||\lambda|| =m$. For the empty partition $\varnothing$, we define $||\varnothing||= |\varnothing| =0$.

In what follows we will have to consider sums over ordered partitions of a certain length. By $\sum_{|\lambda| =k} f(\lambda)$ we mean that we sum the values of $f(\lambda)$ as $\lambda$ ranges over all ordered partitions of $\{1, \dots, k\}$.

\begin{defn}
    . Let $X$ be a quasiprojective variety over $\FF$. Let $S$ be a nonempty finite set. Let $\lambda = (A_1, A_2, \dots, A_m) $ be an ordered partition of $S$.  We define $\Sym^{\lambda} (X) := \prod_{i=1}^m \Sym^{|A_i|}(X)$. 
     
     There is a quotient map $\pi: X^{|S|} \to \Sym^{\lambda}(X)$
     We define the $\lambda$ configuration space of $X$,  $w_{\lambda}(X) \subseteq \Sym^{\lambda}(X)$ to be $\pi(X^{|S|}\setminus \Delta)$, where $\Delta$ is the big diagonal.

     We define $\Sym^{\varnothing}(X) =w_{\varnothing}(X)= \Spec\textrm{ } \FF$.
\end{defn}

If $\lambda$ is an ordered partition and $X$ is a quasiprojective variety, any $Z \in w_{\lambda}(X)$ defines a reduced zero dimensional subscheme of $X$ which we will denote $\supp(Z)$. For any subvariety $Y \subseteq X$, we say $Z \subseteq Y$ if $\supp(Z) \subseteq Y$.

There is a forgetful covering map $w_{\lambda}(X) \to w_{|\lambda|}(X)$ defined by $Z \mapsto \supp(Z).$

\begin{defn}
    Let $X$ be a quasi projective variety over $\FF$. Let $\LL$ be a line bundle on $X$. We define $$W_{\ge \lambda}(X) = \{(f,Z) \in H^0(X,\LL) \times w_{\lambda}(X)| Z \subseteq \textrm{Sing}(f)\}.$$
    For $N \ge 0$, we define $$W_{\lambda, \ge N}(X):= \{(f,Z) \in W_{\ge \lambda}(X)| |\textrm{Sing}(f)| \ge N\}.$$

    We define $$W_{ \lambda}(X) := \{(f,Z) \in W_{\ge \lambda}(X)| |\lambda|=|\textrm{Sing}(f)|\}.$$

\end{defn}
   Our notation does not involve the line bundle $\LL$. We do not believe this will cause confusion as the line bundle $\LL$ will be fixed throughout. We note that $W_{\varnothing}(X) = U(X, \LL)$ the discriminant complement.
\begin{defn}
    Let $S_1 \subseteq S_2$ be two nonempty finite sets. Let $\lambda_i = (A_1^i, A_2^i, \dots, A^i_{m_i})$ be ordered partitions of $S_i$. We say $\lambda_1$ is a subpartition of $\lambda_2$, denoted $\lambda_1 \subseteq \lambda_2$, if there is an increasing function $f:\{1, \dots, m_1\} \to \{1, \dots, m_2\}$ such that $A_j^1 \subseteq A_{f(j)}^2$ for all $j \in \{1, \dots, m_1\}$. If this is the case we may form an ordered partition of $S_3= S_1 \setminus S_2$ as follows. We define $B_i = A_i^2 - A_j^1$ if $i = f(j)$ for $j \in \{1, 2 \dots, m_1\}$, otherwise we define $B_i = A_i$. We define $\lambda_3 $ to be the partition corresponding to the tuple $(B_1, \dots B_{m_2})$ after deleting empty sets.
\end{defn}
In what follows we will need to sum over all subpartitions of a given partition $\lambda$. By $\sum _{\alpha \subseteq \lambda} f(\alpha)$ we mean the sum of $f(\alpha)$ where $\alpha$ ranges over all subpartitions of $\lambda$.

\begin{defn}
    We assume $X$ is a quasi projective variety and $\LL$ is a very ample line bundle on it.
    We define the \emph{forced positive dimensional singularities} subset of $w_{\lambda}(X)$ denoted $w_{\lambda}^p(X)$ by $$w_{\lambda}^p(X) =\{ Z \in w_{\lambda}(X) | \textrm{if } f \in H^0(X, \LL), Z \subseteq \textrm{Sing}(f), \textrm{ then }\dim \textrm{Sing}(f) \ge 1  \}.$$ Let $w_{\lambda}^n(X) = w_{\lambda}(X) \setminus w_{\lambda}^p(X).$

    Let $ W_{\ge \lambda}^p(X) = \{(f,Z) \in W_{\ge \lambda}(X) | Z \in w_{\lambda}^p(X)\}.$   Let $ W_{\ge \lambda}^n(X) =  W_{\ge \lambda}(X) \setminus W_{\ge \lambda}^p(X).$ Let $ W_{ \lambda, \ge N}^p(X) = \{(f,Z) \in W_{ \lambda ,\ge N}(X) | Z \in w_{\lambda}^p(X)\}.$ Let $ W_{ \lambda, \ge N}^n(X) =  W_{ \lambda, \ge N}(X) \setminus W_{ \lambda, \ge N}^p(X).$
\end{defn}

 \begin{prop}
     Let $X, \lambda$ be as above. The set $w_\lambda^p(X)$ is a constructible subset of $w_{\lambda}(X).$
 \end{prop}
\begin{proof}
    Let $j$ be a positive integer. Let $$w_{\lambda,j}(X) =\{(Z, f_1, \dots, f_j)| Z \in w_{\lambda}(X), f_i \in H^0(X,\LL), Z \subseteq \Sing(f_i)\}.$$ There is a natural forgetful map $w_{\lambda,j}(X) \to w_{\lambda}(X)$. Let $$w_{\lambda,j}^p(X) = \{(Z, f_1 ,\dots f_j)| f_i \textrm{form a basis for } H^0(X, I^2_Z\otimes \LL)\}.$$ Clearly $w_{\lambda}^p(X)$ is the union of the images of $w_{\lambda,j}^p(X)$. Hence it suffices to establish that the $w_{\lambda,j}^p(X)$ are constructible. 

    Let $E_{\lambda,j}(X) = \{(P, Z, f_1, \dots, f_j) \in X \times w_{\lambda^j}(X)| P \in \cap \Sing(f_i)\}.$ There is an obvious forgetful map $\pi: E_{\lambda_j}(X) \to w_{\lambda,j}(X)$. Consider a flattening stratification of $S_1 ,\dots S_k$ $w_{\lambda,j}$ with respect to $\pi.$ We note that $(Z, f_1, \dots, f_j) \in w_{\lambda,j}^p(X)$ iff the $f_i$ form a basis of $H^0(X, I^2_Z\otimes \LL)$ and if $\pi^{-1}((Z, f_1, \dots, f_j))$ is positive dimensional. However the dimension of $\pi^{-1}((Z, f_1, \dots, f_j))$ is constant on each stratum. Hence $w_{\lambda,j}^p(X)$ is just the union of a collection of strata, intersected with the set of $(Z, f_1 ,\dots, f_j)$ where the $f_i$ form a basis of $H^0(X, I^2_Z\otimes \LL)$. However the latter is clearly a constructible set as is each stratum. Hence $w_{\lambda,j}^p(X)$ is the union of an intersection of constructible sets and is hence constructible. 
    \end{proof}

\begin{rem}
    The p and n in the notation $w_{\lambda}^p, w_{\lambda^n}$ is intended to mean positive singularities and not positive dimensional singularities. These variants of the $w_{\lambda}$ are largely used in the final section of this paper and are necessary because we will often want to disregard elements of $w_{\lambda}^p$ for various technical reasons. The necessity of using the $w_\lambda^n$ is almost exclusive to Proposition \ref{geN0} and the reader is advised to not pay too much attention to it at a first reading.
    \end{rem}

\begin{prop}\label{propVW} Let $N$ be a positive integer.
The following equations are true, where both sides are treated as part of $\M_{\Lb}$
    $$W_{\varnothing} = \sum_{|\lambda| \le N} (-1)^{||\lambda||}W_{\ge \lambda} - \sum_{|\lambda| \le N} (-1)^{||\lambda||} W_{\lambda, \ge N+ 1}.$$ 

    $$W_{\varnothing} = \sum_{|\lambda| \le N} (-1)^{||\lambda||}W^n_{\ge \lambda} - \sum_{|\lambda| \le N} (-1)^{||\lambda||} W^n_{\lambda, \ge N + 1}.$$ 
    
\end{prop}
\begin{proof}
    The first equation follows from the proof of Proposition 3.7 in \cite{VW}.
    The second follows immediately from the first and the fact that 
    $$W_{\ge \lambda} - W_{\lambda, \ge N- |\lambda| + 1} = W^n_{\ge \lambda} - W^n_{\lambda, \ge N + 1}.$$
\end{proof}

\begin{prop}\label{pm}
    Let $Z$ be a constructible subset of $w_{[n]}(X)$. Let $\pi_{\lambda}: w_{\lambda}(X) \to w_{[|\lambda|]}(X)$ be the covering map. Let $\phi : \M \to \M_{Hdg}$ denote the  specialisation homomorphism.
    Then $$\phi(\sum_{|\lambda| = n}(-1)^{||\lambda||}(\pi_{\lambda}^{-1}(Z))) = H^*_c(Z , \pm \QQ)$$, where $\pm \QQ$ is the restriction of the alternating sheaf on $w_{[n]}(X)$.

    Also, for $k >3$, $H^*_c(w_{k}(\PP^1), \pm \QQ) =H^*_c(w_{k}(\A^1), \pm \QQ) =0.$
\end{prop}
\begin{proof}
    Let $S_{\lambda}$ be the local system $(\pi_{\lambda})_* \QQ$. Then by \cite{DH}  $\sum_{|\lambda| = n} (-1)^{||\lambda||} S_ {\lambda} = \pm \QQ$ as a virtual representation. Thus, $\sum_{|\lambda|= n} (-1)^{||\lambda||}H^*_c(Z, S_{\lambda}) = H^*_c(Z , \pm \QQ)$. However $H^*_c(Z, S_{\lambda}) = H^*_c(\pi_{\lambda}^{-1}(Z), \QQ)$. This completes the proof.
    The equality $H^*_c(w_{k}(\PP^1), \pm \QQ) =H^*(w_{k}(\A^1), \pm \QQ) =0,$ is established in Lemma 2 of \cite{V}.
\end{proof}

We define $$w_{\lambda}^l(X) = \{Z \in w_{\lambda}(X) | h^0(X, I^2_Z(d)) + h^0(N(Z), \OO(d)) = l + h^0(X, \OO(d)) \}.$$

We note that the $w_{\lambda}^l(X)$ are constructible, this follows immediately from the fact that $\{Z \in w_{\lambda}(X)| h^0(X, I^2_Z(d)) + h^0(N(Z), \OO(d)) \ge l + h^0(X, \OO(d))\}$ is closed, which is a consequence of semicontinuity of cohomology.

\begin{prop}
   Let $X$ be a smooth projective variety of dimension $n$. Let $\lambda$ be an ordered partition. Then,
   $$[W_{\ge \lambda}(X)] = H^0(X, \LL)\sum_{l= 0}^{\infty} [w_{\lambda}^l(X)] \Lb^{l-|\lambda|(n +1)}.$$

\end{prop}
\begin{proof}
    Let $\pi: W_{\lambda}(X) \to w_{\lambda}(X)$ be the projection map. We note that $w_{\lambda}(X)$ is the disjoint union of $w_{\lambda}^l(X)$. Therefore $W_{\lambda}(X)$  is the disjoint union of $\pi^{-1}(w_{\lambda}^l(X)),$ and hence $[W_{\lambda}(X)]= \sum_{l=0}^{\infty} [\pi^{-1}(w_{\lambda}^l(X))].$

    The map $\pi: \pi^{-1}(w_{\lambda}^l(X)) \to w_{\lambda}^l(X) $ is a vector bundle of dimension $h^0(X, \LL) - |\lambda|(n+1) + l$    by Grauert's theorem (see for instance p. 288 Cor. 12.9 \cite{H} ). Hence $$[\pi^{-1}(w_{\lambda}^l(X))] = H^0(X, \LL)[w_{\lambda}^l(X) ]  \Lb^{l-|\lambda|(n +1)}.$$ 
\end{proof}
\begin{prop}\label{h1bound}
Let $X$ be a smooth projective variety of dimension $n$. Let $C \subseteq X$ be a reduced curve of degree $e$. Let $Z \subseteq C$ be a reduced set of points. Let $d \gg \deg C$. Then $$h^1(X, I^2_Z(d)) \le \max(2 |Z | -ed , 0) + (n-1)\max( |Z| -ed, 0) - k.$$
\end{prop}
\begin{proof}
    This follows immediately from Corollary 1.1.25 in \cite{La}.
\end{proof}
 Let $\psi_{e,d}(l)$ denote the minimum $m$ such that $l \ge  \max(2 m -ed , 0) + (n-1)\max( m -ed, 0) - k.$
\begin{prop}\label{boundformainrate}
    Let $X$ be a smooth projective variety of dimension $n \ge 2$. Let $l \ge 1$.   Let $d \gg e >0$.
    Let $k$ be as in Proposition \ref{h1bound}.

    Let $\lambda$ be an ordered partition such that $e \ge |\lambda| /d$. Then $$\dim w_{\lambda}^l(X) \le \max_{e> 0}(|\lambda|(n) - (n-1)\psi_{e,d}(l) ) +K$$ where $K = \max_{r \le e} \dim Chow_r(X)$ where $Chow_r(X)$  is the Chow variety of reduced curves of degree $r$ in $X$. Let $r_0$ denote the minimal degree of any curve in $X$.

    Let $(nr_0+1)d> N > r_0n d $ Then, there exists $C$ such that $$\dim W_{\lambda, \ge N }^l(X) \le h^0(X, \OO(d))  - r_0nd  +K .$$
\end{prop}
\begin{proof}
    By Lemma \ref{contcurve} any $Z \in w_{\lambda}^l(X)$ must be of the form $Z' \cup Z''$ where $Z'= Z \cap C$ for some curve $C$ of degree $e'\le e$. Furthermore $l =h^1( X , I^2_Z(d)) = h^1(X, I^2_{Z'}(d))$. By Proposition \ref{h1bound}, this implies that $|Z'| \ge \psi_{e,d}(l) $.  
    
As a result, $$\dim w_{\lambda}^l(X) \le \max_{|\lambda|/d > e> 0}(|\lambda|(n) - (n-1)\psi_{e,d}(l) ) +\max_{|\lambda|/d > e> 0} \dim Chow(e).$$
    
\end{proof}

We now prove Theorem \ref{mainrate}, The essential idea is as follows- we use the techniques of \cite{VW} to argue that the difference between $\frac{U(X,\LL)}{H^0(X, \LL)} - \zeta^{-1}(\dim X +1)$ is  a certain sum of contributions corresponding to configurations of points contained in $w_{\lambda}^l(X)$ where $l\ge 1$. We then use Proposition \ref{boundformainrate} to bound this contribution.
\begin{proof}[Proof of Theorem \ref{mainrate}]
    By Proposition \ref{propVW} $$U(X, \LL) = W_{\varnothing}(X) =  \sum (-1)^{||\lambda||}(W_{\ge \lambda}(X) - W_{\lambda, \ge N +1 }).$$ By Proposition \ref{boundformainrate}  the terms $W_{\lambda, \ge N  +1 }) = 0$ up to dimension $-  m(d)(\dim X) + C$. Also for $ ||\lambda|| \le N , l \ge 1$, $w_{\lambda}^l$ lies in $F$. Hence we may apply Proposition and conclude that $W_{\ge \lambda} = w_{\lambda}\Lb^{-|\lambda|(n+1)} $ upto dimension  $-  m(d)(\dim X) + C$.

    This implies that $W_{\varnothing}(X) = \sum_{\lambda} (-1)^{||\lambda||} w_{\lambda}(X)\Lb$, up to  dimension $-  m(d)(\dim X) + C$.. But by Proposition 3.7 of \cite{VW} the right hand side of this last equation equal $Z_X^{-1}(\dim X +1)$.

\end{proof}
\section{Proof of Theorem \ref{maindim}}
In this section we will prove Theorem \ref{maindim}. The proof of Theorem \ref{maindim} is similar to that of Theorem \ref{mainrate} in that we use Proposition \ref{propVW}  and provide bounds on the dimensions of $w^l_\lambda(X)$. However our assumptions are different. This section is strictly speaking independent of Sections 2 and 3 , and we do not use any results on Cayley-Bacharach sets.

We define the projective hull of a subscheme $Y \subseteq \PP^n$, denoted $phull(Y)$ to be the intersection of all projective subspaces of $\PP^n$ containing $Y$.

\begin{prop}\label{affPn}
   Let $n, m \ge 0$. Let $Z \subseteq \PP^n$ be a reduced zero dimensional subscheme. SUppose $|Z| \ge m+2$. Suppose that $phull(Z)$ in $\PP^n$ is $m$ dimensional.  Let $d \ge 3$. Then the kernel of the restriction map $H^0(\PP^n, \OO(d)) \to H^0(N(Z),\OO(d)) $ is of codimension $\ge (m +1)(n+1) + n- 1$. 
\end{prop}
\begin{proof}

    We claim that it suffices to prove the Proposition in the special case when $Z = \{P_0, \dots, P_m, P\}$ where the projective hull of the $P_i$ is $m$ dimensional, and $P$ is in $phull(\{P_0 , \dots, P_m\}).$ To see this, first note that if $Z' \subseteq Z$ the  kernel of $H^0(\PP^n, \OO(d)) \to H^0(N(Z),\OO(d)) $  is contained in the kernel of $H^0(\PP^n ,\OO(d)) \to H^0(N(Z'), \OO(d)).$ As a result, to establish the Proposition for $Z$, it suffices to establish it for any subset $Z' \subseteq Z$. We then note that given any $Z$ whose projective hull is $m$ dimensional, we can find a $Z_{ind} \subseteq Z$ such that $|Z_{ind}| = m+1$ and $phull(Z_{ind}) = phull(Z)$. We now consider any point $P  \in  Z \setminus Z_{ind}$ and note that $Z' = Z_{ind} \cup \{P\}$ is a subset of $Z$ satisfying the conditions of the claim and that to prove the Proposition for $Z$ it suffices to prove it for $Z'$.

    In light of the above paragraph, we now assume $Z = \{P_0, \dots, P_m, P\},$ where $phull(P_0, \dots, P_m)$ is $m$ dimensional and contains $P$. We may further assume that $P_i = [0: 0 \dots: 0: 1: 0 \dots 0],$ here the $1$ is in the $i+1^{st}$ position and $P = [x_0: \dots x_m : 0 \dots : 0].$ This is because there is an element  $g \in PGL_{n+1}(\FF)$ satisfying $g(P_i) =  [0: 0 \dots: 0: 1: 0 \dots 0].$

    We now claim that it suffices to prove the proposition in the case where $P = [1 :x : 0 \dots, : 0]$. to see this, note that we may assume without loss of generality that if $P = [x_0: x_1: \dots, x_m: 0 \dots : 0],$ $x_0 \neq 0$ and $x_1 \neq 0$. For $t \in \A^1$, let $Z_t$ be the reduced zero dimensional subscheme $\{P_1, \dots P_m. P_t\}$, where $P_t = [x_0: x_1: t x_2 ,\dots tx_m: 0 \dots 0].$ The $Z_t$ form a family of subschemes over $\A^1,$ with the original subscheme $Z$ being equal to $Z_1$. Furthermore, given $t \neq 0$, there is a $g \in PGL_{n+1}(\FF)$ such that $g(Z_t) = Z_1$. This has the consequence that the dimension of the kernel of the restriction of $H^0(\PP^n,\OO(d)) \to H^0(N(Z_t) , \OO(d))$ is independent of $t$ as long as $t \neq 0$. The kernel  of  $H^0(\PP^n,\OO(d)) \to H^0(N(Z_t) , \OO(d))$  is isomorphic to $H^0(\PP^n , I^2_{Z_t}(d)).$ Semicontinuity now implies that $h^0(\PP^n , I^2_{Z_0}(d)) \ge h^0(\PP^n , I^2_{Z_1}(d)).$ As a result it suffices to prove the proposition for $Z_0$. However $Z_0$ is of the required form.

    To complete the proof, we directly compute  the bound on the codimension of the kernel when $Z$ is of the form $\{P_0, \dots, P_m, P\}$ where $P_i = [0: 0 \dots: 0: 1: 0 \dots 0],$  and $P = [1: x:0 \dots: 0]$. Let $f$ be a section of $H^0(\PP^n ,\OO(d))$, we may think of $f$ as a homogenous polynomial. Then  requiring $f$ to vanish to second order at $P_i$ implies that the coefficients of $X_i^d$ and $X_jX_i^{d-1}$ are zero. Let $a_{k,r}$ be the coefficient to $X_kX_1^r X_0 {d - 1 -r}$. Requiring that the partial derivatives of  $f$ vanish at $P$ is equivalent to requiring that for $k \in \{2, \dots n\}$, $$
    \sum_{r=0}^{d} x^ra_{k,r} = 0.$$
    We note that the above linear relations are linearly independent and hence the kernel of $H^0(\PP^, \OO(d)) \to H^0(N(Z), \OO(d))$ is of codimension at least $(m+1)(n+1) + n-1$.

\end{proof}

The above proof gives us a little more information than the statement of Proposition \ref{affPn}. We record this information in the following Proposition.

\begin{prop}\label{linindPm}
    Let $m,n \ge 1$. Let $d \ge 3$
   Let $P_i  = [0: \dots 1: : 0 \dots : 0] \in \PP^n$, where the $1$ is in the $i+1^{th}$ position. Let $P \in phull\{P_0 ,\dots, P_m\}$ Let $V_i^j$ be any basis of $T_{P_i}(\PP^n)$. Let $V^j$ be a basis of $T_P \PP^n$.

Then requiring a polynomial $f \in H^0(\PP^n,\OO(d))$ to satisfy $$f(P_i) =0, \frac{\partial}{\partial V_i^j}(f)(P_i) = 0, \frac{\partial}{\partial v^j}f(p) =0 ,$$ imposes $(m+1)(n+1) + n-1$ linearly independent conditions on $f$. 
\end{prop}
\begin{proof}
    We omit the proof as this is already present in the proof of Proposition \ref{affPn}.
\end{proof}

\begin{prop} \label{affX}
   Let $n ,m\ge 0$. Let $Z \subseteq X \subseteq \PP^n$ be a reduced set of points, $m+1 <|Z| < \dim X $.  Let $d\ge 3$. Let $\dim phull(Z) = m$. Then the kernel of the restriction map $H^0(X, \OO(d)) \to H^0(N(Z),\OO(d))$ is of codimension $\ge (m +1)(\dim X +1) + \dim X -1$.
\end{prop}
\begin{proof}
    This proof is similar to that of Proposition \ref{affPn} .
    It suffices to prove the proposition in the case where $Z = \{P_0, \dots, P_m, P\}$ where $P_i= [0: \dots : 0 :1:  \dots: 0]$ where there is a $1$ in the $i+1$th position. This is by an argument identical to that in Proposition \ref{affPn}.

    Let $V_i^j$ be a basis of $T_{P_i}X$. Let $v^j$ be a basis of $T_P X$.It suffices to prove that requiring a polynomial $f$ to satisfy at $f(P_i) =0, $ for all $i$, $f(P_i)=0$ and that $\partial_{v_i^j}(f)(P_i)$, $\partial_{v_i}(f)(P)$ impose linearly independent conditions on $f$. However, this follows from Proposition \ref{linindPm}, since in that proof we show that a larger set of linear equations impose linearly independent conditions on $f$.
\end{proof}

\begin{prop}\label{bounddimbigdim}
    Let $n \ge 0$. Let $X$ be a smooth projective variety of dimension $n$. Let $\lambda$ be an ordered partition, $|\lambda| \le \sqrt{n}/3 $. Let $l \ge 1$. Then $$\dim w_{\lambda}^l(X) + l - (n + 1) (|\lambda|) \le -\sqrt{n}/{3}$$ and $$\dim W_{\lambda, \ge N}^l(X) \le h^0(X, \OO(d)) - \sqrt{n}/{3}.$$
    In addition we have $\dim W_{\lambda, \ge N}^0(X) \le h^0(X, \OO(d)) - \sqrt{n}/{3}.$
\end{prop}
\begin{proof}
Let $ l\ge 1$.
Let $Z \in w_{\lambda}^l(X)$. Then by definition, $h^0(X, I^2_Z(d)) = h^0(X, \OO(d)) - |Z|(\dim X +1) +l$. Let $\dim phull(Z) = m$. Then by Proposition \ref{affX}, $|Z|(\dim X + 1) - l \ge (m+1) (\dim X + 1) - l$. 
     But a simple dimension count shows that the space of $|\lambda|$ points in $X$ contained in an $m$ dimensional subspace is of dimension $ \le (m+1) \dim X + (|\lambda| - m-1)m$.

    This implies that $\dim w_{\lambda}^l(X) \le \max_m (m+1) \dim X + (|\lambda| - m-1)m.$

    We will now bound the quantity $(m+1) \dim X + (|\lambda| - m-1)m +l - (\dim X + 1)|\lambda|$. By the above we have 
    $$ (m+1) \dim X + (|\lambda| - m-1)m +l - (\dim X + 1)|\lambda| $$ 
    
    $$\le  (\dim X +1)|\lambda| - l - \dim X + 1 + (|\lambda| - m -1 )m + l - (\dim X +1) |\lambda| $$

    $$ = - \dim X + 1 + (|\lambda| - m -1)m \le - \dim X + 1 + (|\lambda|)^2 \le - \dim X + 1 + \frac{\dim X}{9}$$ 
    
    $$ \le -\frac{8}{9}\dim X + 1 \le -\frac{\sqrt{\dim X}}{3} .$$

    As a result  $$\dim w_{\lambda}^l(X) + l - (n + 1) (|\lambda|) \le -\sqrt{n}/{3}$$. The inequality  $$\dim W_{\lambda, \ge N+1}^l(X) \le h^0(X, \OO(d)) - \sqrt{n}/{3}$$ follows similarly to the above for $l \ge 1$ and is immediate if $l =0$.

    This implies the result.
\end{proof}
We note here that the proof of the above Proposition gives us a better bound on the dimension of $w_{\lambda}^l(X)$ than we claimed in the result. However we are unable to utilise this improved bound. The sticking factor is that we only have the bound $\dim W_{\lambda, \ge N}^0(X) \le h^0(X, \OO(d)) - N$, which we believe to be tight.
\begin{proof}[Proof of Theorem \ref{maindim}]
   Let $N = \sqrt{n} /3$.
   By Proposition \ref{propVW}  $$\frac{W_{\varnothing}(X)}{H^0(X, \OO(d))} = \sum_{|\lambda| \le N} (-1)^{||\lambda||}\frac{ W_{\ge \lambda} (X)}{H^0(X,\OO(d))} - \frac{W_{\lambda, \ge N  +1}(X)}{H^0(X,\OO(d)}.$$

   By Proposition \ref{propVW} the right hand side equals $$\sum_{|\lambda| \le N} w_{\lambda}^0(X)\Lb^{-|\lambda| (\dim X + 1)} + \sum_{|\lambda| \le N}w_{\lambda}^l(X)(\Lb^{-|\lambda|) (\dim X + 1) + l} - \sum_{|\lambda| \le N}\frac{ W_{\lambda, \ge N  +1}(X)}{H^0(X,\OO(d))}.$$  By Proposition \ref{propVW} the first term equals $Z_X^{-1}(\dim X +1) -\sum_{|\lambda| \le  N} w_{\lambda}^l(X)\Lb^{-|\lambda| (\dim X + 1)} -\sum_{|\lambda| > N} (-1)^{||\lambda||}w_{|\lambda}(X) \Lb^{-|\lambda| (\dim X + 1)} = Z_X^{-1}(\dim X +1)  $ up to dimension $- N$. By Proposition \ref{bounddimbigdim} the latter two terms are of dimension $\le -N$ which implies the result.
\end{proof}

\section{Proof of Theorem \ref{mainPn}}
In this section we prove Theorem \ref{mainPn} . The proof of Theorem \ref{mainPn} is similar to that of Theorems \ref{mainrate} and \ref{maindim} ,  albeit a bit more involved.

In this section we will exclusively work with $X = \PP^n$. As a result we will drop the $\PP^n$ from our notation, i.e. we will refer to $w_{\lambda}(\PP^n)$ as $w_{\lambda}$, $W_{\ge \lambda}(\PP^n) = W_{\ge \lambda}$ etc. However if the argument of $w_{
\lambda},$  is not $\PP^n$ we will explicitly write down what the space is.

The following basic proposition will be used  repeatedly in this section.
\begin{prop}\label{h1N}
    Let $C$ be a smooth curve in $\PP^n$. Then for $d \gg \deg C$,
    $$h^0(  N (C), \OO(d)) =  h^0(C, \OO(d)) + h^0(C, N_C^*(d)).$$
    
\end{prop}

\begin{proof}
    Let $I$ denote the ideal sheaf of $C$.    Consider the following chain of inclusions of sheaves on $\PP^n$-
    $$ 0 \subseteq I^2(d) \subseteq I(d) \subseteq \OO(d).$$ We will break the above exact sequence into short exact sequences. We have a short exact sequence 
    $$0 \to I^2(d) \to I(d) \to I/I^2(d) \to 0.$$ Under the assumption that  $d\gg \deg(C)$, Serre vanishing implies that $H^1(\PP^n , I^2(d)) =0$. As a result, $$h^0(\PP^n, I(d)) = h^0(\PP^n,I^2(d)) + h^0(\PP^n, I/ I^2 (d)).$$
    
    We also have the short exact sequence $0 \to I(d) \to \OO(d) \to \OO/ I (d).$ Since $d \gg \deg(C)$, $$h^0(\PP^n, \OO(d)) = h^0(\PP^n, I(d)) + h^0(\PP^n, \OO / I (d)).$$

    Finally we have the short exact sequence $0 \to I^2(d) \to \OO(d) \to \OO/ I^2 (d).$ Since $d \gg \deg(C)$, $$h^0(\PP^n, \OO(d)) = h^0(\PP^n, I^2(d)) + h^0(\PP^n, \OO / I^2 (d)).$$

    As a result, we have $$h^0(\PP^n, \OO(d) / I^2(d)) = h^0(\PP^n , \OO/ I(d)) + h^0(\PP^n, I / I^2(d)).$$
    But we now note that $\OO/ I^2$ is the push forward of the structure sheaf of $N(C)$ and $\OO/ I$ is the pushforward of the structure sheaf of $C$. As a result $$h^0(\PP^n, \OO(d) / I^2(d)) = h^0(N(C), \OO(d))$$ and $$h^0(\PP^n, \OO(d) / I(d)) = h^0(C, \OO(d)).$$ Finally, we note that $I/I^2$ is the pushforward of the conormal sheaf of $C$ and as a result $h^0(\PP^n , I/I^2(d)) = h^0(C, N^*(d))$. This gives us the required result.

\end{proof}

\begin{defn}
    Let $\textrm{Chow}(r)$ be the Chow variety of degree $d$ curves in $\PP^n$. Let $M \subseteq \PP^n$ be a constructible subset of $Chow(r)$. Let $\lambda$ be an ordered partition. We define $w_{\lambda}^{M, l}  \subseteq w_{\lambda}^l$ as consisting of all $Z \in w_{\lambda}$ such that there exists $C \in M$ satisfying $$h^1(\PP^n, I^2_Z(d)) = h^1(\PP^n ,I^2_{Z \cap C}(d))$$ and furthermore there is no curve $C'$ of degree less than $r$ such that $$h^1(\PP^n, I^2_Z(d)) = h^1(\PP^n ,I^2_{Z \cap C'}(d)).$$ 
    \begin{prop}
        Let $d,r,l \ge 0$. Let $\lambda$ be an ordered partition. Let $M$ be a constructible subset of $Chow(r)$.
         Assume $d \gg r .$ The subset $w_{\lambda}^{M, l} $ is constructible and hence defines an element in $\M$.
    \end{prop}
     \begin{proof}
         Let $$\tilde w_{\lambda}^{M, l}  := \{(Z, C) \in w_{\lambda}^{M, l}  \times M | h^1(\PP^n, I^2_Z(d)) = h^1(\PP^n ,I^2_{Z \cap C}(d)) \}.$$
         We note that since $d \gg r$, by Proposition \ref{contcurve} for any $Z \in  w_{\lambda}^{M, l}$, there exists a unique $(Z, C) \in  \tilde w_{\lambda}^{M, l}$ mapping onto it under the projection, thus it suffices to establish that $\tilde w_{\lambda}^{M, l}$ is constructible. Consider $X = w_{\lambda} \times M \times \PP^n$. The variety $X$ has two ideal sheaves on it $\Fcal_1$ and $\Fcal_2$ corresponding to the subvarieties $E_1 = \{(Z,C,P) \in \tilde w_{\lambda}^{M, l} \times \PP^n | P \in Z\}$ and $E_2 =\tilde w_{\lambda}^{M, l} \times \PP^n | P \in Z \cap C \}$. Let $\pi: \tilde w_{\lambda}^{M, l} \times \PP^n \to \tilde w_{\lambda}^{M, l}$ denote the projection. Now we may consider a stratification $S_i$ of $\tilde w_{\lambda}^{M, l}$ such that the sheaves $\pi_* \Fcal_i(d)$ are flat on each stratum. The semicontinuity theorem then implies that one may further refine the stratification $S_i$ to a stratification $S_i'$ on which $h^1(\PP^n , I^2_{Z \cap C}(d))$ and $h^1(\PP^n, I^2_Z(d))$ are constant. We then note that $ \tilde w_{\lambda}^{M, l} $ is precisely the subset of $X$ where $h^1(\PP^n , I^2_{Z \cap C}(d))$ and $h^1(\PP^n, I^2_Z(d))$ are equal to $l$. Thus $\tilde w_{\lambda}^{M, l}$ is a union of strata and is thus constructible.
     \end{proof}
    Let $\pi: W_{\ge \lambda} \to w_{\lambda}$ denote the natural projection.
    Let $W_{ \ge \lambda}^{M,l}(\PP^n) =\pi^{-1}(w_{\lambda}^{M,l}) $. Let $W_{\lambda}^{M,l} = W_{\ge \lambda}^{M,l} \cap W_{\lambda}$.
    We define $$\bar w_{\lambda}^{M,l}= \{Z \in w_{\lambda} | \exists C \in M, h^1(\PP^n, I^2_{Z \cap C}(d)) = l\}.$$
    We note that $$ w_{\lambda}^{M,l} \subseteq \bar w_{\lambda}^{M,l}.$$
\end{defn}

We start by decomposing the varieties $w_{\lambda}^l$ into pieces.
\begin{prop}
Let $l \ge 1$.  $$w_\lambda^l  = \sum_{r=1}^{\infty} w_{\lambda}^{\textrm{Chow}(r),l}.$$
\end{prop}
\begin{proof}
 We note that the sum is finite, for $r$ sufficiently large  $w_{\lambda}^{\textrm{Chow}(r),l}$ is empty.

    It then suffices to prove that the collection $w_{\lambda}^{\textrm{Chow}(r),l}$ partitions $w_\lambda^l$. This follows from Proposition \ref{contcurve}.
\end{proof}

\begin{prop}
     Let $e \ge 1$. Let $\mathcal{A}$  be a finite collection of constructible subsets defining a partition of $\textrm{Chow}(e).$ Then $w_{\lambda}^{\textrm{Chow}(e),l} = \sum_{M \in \mathcal{A}}w_{\lambda}^{M,l}.$
\end{prop}
\begin{proof}
    This follows from the fact that a partition of $\textrm{Chow}(e)$ induces a partition of $w_{\lambda}^{\textrm{Chow}(e),l}$.
\end{proof}
We now prove a dimension bound on several of the $w_{\lambda}^{M,l} $.

\begin{prop}\label{ge4}
    Let $e \ge 4$. Let $M \subseteq \textrm{Chow}(e) $ be a constructible set. Let $n \ge 2, l \ge 1 $. Then there exists $\epsilon >0$  such that, $$\dim (W_{\ge\lambda}^{M,l} - W^{M,l}_{\ge\lambda,  \ge N +1}) \le h^0(\PP^n ,\OO(d)) - (\frac{3n}{2} + \epsilon)d.$$

    Furthermore $\dim w_{\lambda}^{M,l} \le |\lambda| (n)- l -(\frac{3n}{2} + \epsilon)d$.
\end{prop}
\begin{proof}
    Let $M \subseteq \textrm{Chow}(e)$. We will first prove that $\dim w_{\lambda}^{M,l} \le  |\lambda| (n)- l -(\frac{3n}{2} + \epsilon)d$. We note that $\dim w_{\lambda}^{M,l} \le \dim \tilde w_{\lambda}^{M,l}$. By Proposition \ref{h1bound}  any $(Z,C) \in w_{\lambda}^{M, l}$ must have $|Z \cap C| \ge \psi_{e,d}(l)$.

    Thus $$\dim w_{\lambda}^{M, l} \le \dim \tilde w_{\lambda}^{M,l}  \le \dim M + (|\lambda| - \psi_{e,d}(l)) (n) + \psi_{e,d}(l).$$

    An inspection of the function $\psi_{e,d}$ then implies that for $e \ge 4$ and $d \gg 0$, $\dim w_{\lambda}^{M,l} \le |\lambda| (n)- l -(\frac{3n}{2} + \epsilon)d.$
     Since $W_{\ge \lambda,l}^{M} =w_{\lambda}^{M,l} H^0(\PP^n ,\OO(d)) \Lb^{-|\lambda|(n+1) +l}$, the result follows.

\end{proof}
\begin{prop}\label{barminus}
Let $d_0 \ge 1$.
    Let $M \subseteq Chow(d_0)$ be a constructible set. Then $w_{\lambda}^{M, l} \subseteq \bar w_{\lambda}^{M,l}$. Furthermore $\bar w_{\lambda}^{M,l} \setminus w_{\lambda}^{M, l}= \cup_{d >d_0, l'>l} w_{\lambda}^{M'(d),l'} ,$ where $M'(d)$ is the subset of $Chow(d)$ consisting of reducible curves that contain a curve in  $M$.
\end{prop}
\begin{proof}
    This follows from Lemma \ref{contcurve}.
\end{proof}
\begin{prop}\label{wdecomp}
    Let $X$ and $Y$ be algebraic varieties over $\FF$. Then $w_{\lambda}(X \amalg Y) = \sum_{\alpha\subseteq \lambda} w_{\alpha}(X) w_{\lambda \setminus \alpha} (Y).$
\end{prop}
\begin{proof}
    This follows immediately from the fact that $w_{\lambda}(X \amalg Y)$ is partitioned into subsets isomorphic to $w_{\alpha}(X) \times w_{\lambda \setminus \alpha} (Y)$ as $\alpha$ ranges over all subpartitions of $\lambda$.
\end{proof}
\begin{prop}\label{polyinL}
   Let $X$ be an algebraic variety. Let $Y$ be a subvariety of $X$. Suppose $[X], [Y]$ are polynomials in $\Lb$. Then for any $\lambda$, $w_{\lambda}(X), w_{\lambda}(Y)$ and $w_{\lambda}(X \setminus Y)$ are polynomials in $\Lb$.
\end{prop}
\begin{proof}
    For $X = \A^n$, $[\Sym^{\lambda}(\A^n)] = \Lb^{|\lambda|n}$. 
    Let $X$ and $Y$ be such that $Y \subseteq X$ and $w_{\lambda}(X)$, $w_{\lambda}(Y)$ can always be expressed as a polynomial in $\Lb$.

    We claim that if $U = X \setminus Y$ then $w_{\lambda}(U)$ can also be expressed as a polynomial in $\Lb$ for all $\lambda$.  Suppose for the sake of contradiction that $\lambda_0$ is an ordered partition such that $w_{\lambda_0}(U)$ is not a polynomial in $\Lb$ and furthermore for all $\lambda' \subseteq \lambda_0$, $w_{\lambda'}(U)$ is a polynomial in $\Lb$.
    Then by Proposition \ref{wdecomp}, $w_{\lambda_0}(U) = w_{\lambda_0}(X) - \sum_{\lambda' \subsetneq \lambda_0} w_{\lambda'} (U) w_{\lambda_0 \setminus \lambda'}(Y)$ which is a polynomial in $\Lb$.

    Hence the class of varieties such that $w_{\lambda}(X)$ is always a polynomial contains $\A^n$ and is closed under differences and unions. This completes the proof of the proposition.
    
\end{proof}

In the next proposition we describe the possible reduced $0$- dimensional subschemes $Z $ contained in a curve $C$ of low degree ($\le 3$) such that $h^1(\PP^n, I^2_Z(d)) = l> 0$.

\begin{prop}\label{h1list}
    \begin{enumerate}
        \item Let $T \subseteq \PP^n$ be a smooth rational normal curve of degree $r \le 3$. Let $Z \subseteq T$ be a  reduced zero dimensional subscheme.  Let $d \gg r$. Then, $h^1(\PP^n, I^2_Z(d)) $ depends only on $|Z|$. Furthermore, $ |h^1(\PP^n, I^2_Z(d)) - \max(2|Z| - rd,0)| + (n-1)\max(|Z|-d,0) \le K$  for some constant $K$.
        \item Let $d \gg e$. Let $C_1 , C_2$ be disjoint reduced curves of degree $\le e$. Let $Z= Z_1 \amalg Z_2$, with $Z_i \subseteq C_i$. Then $h^1(\PP^n,I^2_Z(d)) = h^1(\PP^n,I^2_{Z_1}(d)) + h^1(\PP^n,I^2_{Z_2}(d)) $.
        \item Let $L_1 ,L_2$ be two coplanar lines. Let $Z = Z_1 \amalg Z_2$ such that $Z_i \subseteq L_i$ and $Z \cap L_1 \cap L_2 = \varnothing$. Then $h^1(\PP^n,I^2_Z(d)) $ only depends on $|Z_1|$ and $|Z_2|$ and furthermore, $|h^1(\PP^n,I^2_Z(d)) - h^1(  N(L_1), I^2_{Z_1}(d)) - h^1(  N(L_2), I^2_{Z_2}(d)) | \le 2(n+1)$.

        \item  Let $L_1 ,L_2$ be two coplanar lines. Let $\{p\} = L_1 \cap L_2$ Let $Z = Z_1 \amalg Z_2 \amalg\{p\}$ such that $Z_i \subseteq L_i$ and $Z_i \cap L_1 \cap L_2 = \varnothing$. Then Then $h^1(\PP^n,I^2_Z(d)) $ only depends on $|Z_1|$ and $|Z_2|$ and furthermore, $$|h^1(\PP^n,I^2_Z(d)) - h^1(  N(L_1), I^2_{Z_1}(d)) - h^1(  N(L_2), I^2_{Z_2}(d)) | \le 4(n+1).$$
        \item Let $L, C$ be a line and a conic intersecting at one point, transversally.Let  $Z = Z_1 \amalg Z_2$ such that $Z_1 \subseteq L, Z_2 \subseteq C$ and $Z \cap L \cap C = \varnothing$.
        Then $h^1(\PP^n,I^2_Z(d)) $ only depends on $|Z_1|$ and $|Z_2|$ and furthermore, $$|h^1(\PP^n,I^2_Z(d)) - h^1(N(L), I^2_{Z_1}(d)) - h^1(  N(C), I^2_{Z_2}(d)) | \le 2(n+1).$$
        \item  Let $L, C$ be a line and a conic intersecting at one point, transversally.Let  $Z = Z_1 \amalg Z_2 \amalg \{p\}$ such that $Z_1 \subseteq L, Z_2 \subseteq C$ and $Z_i \cap L \cap C = \varnothing$. Then $h^1(\PP^n,I^2_Z(d)) $ only depends on $|Z_1|$ and $|Z_2|$ and furthermore, $|h^1(\PP^n,I^2_Z(d)) - h^1(  N(L), I^2_{Z_1}(d)) - h^1  (N(C), I^2_{Z_2}(d)) | \le 4(n+1)$.
         \item Let $L, C$ be a line and a conic intersecting at one point, nontransversally.Let  $Z = Z_1 \amalg Z_2$ such that $Z_1 \subseteq L, Z_2 \subseteq C$ and $Z \cap L \cap C = \varnothing$.

         Then $h^1(\PP^n,I^2_Z(d)) $ only depends on $|Z_1|$ and $|Z_2|$ and furthermore, $$|h^1(\PP^n,I^2_Z(d)) - h^1(  N(L), I^2_{Z_1}(d)) - h^1(  N(C), I^2_{Z_2}(d)) | \le 4(n+1).$$
        \item  Let $L, C$ be a line and a conic intersecting at one point, nontransversally.Let  $Z = Z_1 \amalg Z_2 \amalg \{p\}$ such that $Z_1 \subseteq L, Z_2 \subseteq C$ and $Z_i \cap L \cap C = \varnothing$. Then $h^1(\PP^n,I^2_Z(d)) $ only depends on $|Z_1|$ and $|Z_2|$ and furthermore, $|h^1(\PP^n,I^2_Z(d)) - h^1(N(L), I^2_{Z_1}(d)) - h^1(N(C), I^2_{Z_2}(d)) | \le 8(n+1)$. 
        \item Let $L, C$ be a line and a conic intersecting at two points. Let $Z_1 = Z \cap L, Z_2 = Z \cap C$.
        Then $h^1(\PP^n,I^2_Z(d)) $ only depends on $|Z_1|,|Z_2|, |Z_1 \cap Z_2|$ and furthermore, $|h^1(\PP^n,I^2_Z(d)) - h^1(N(L), I^2_{Z_1}(d)) - h^1(N(C), I^2_{Z_2}(d)) | \le 8(n+1)$. 
        \item Let $L_1, L_2, L_3$ be three lines meeting in the following configuration. Let $Z\subseteq L_1 \cup L_2 \cup L_3$. Let $Z_i = Z \cap L_i$. Then $$|h^1(\PP^n,I^2_Z(d)) - \sum h^1(N(L_i), I^2_{Z_i}(d)) | \le 12(n+1).$$ Furthermore, $h^1(\PP^n,I^2_Z(d))$ depends only on the configuration of the lines as well as    $|Z_i|, |Z_i, \cap Z_j|, |Z_i \cap Z_j \cap Z_k|$.
        \item Let $E$ be a plane cubic in $\PP^n$. Let $ Z \subseteq \PP^n$.  Then $h^1(\PP^n,I^2_E(d)) = h^1(E, \OO(d) - 2 Z) + h^1(E , N^* \otimes (\OO(d) - Z))$. Since $E$ is a complete intersection, $N^*$ splits and the term $ h^1(E , N^* \otimes (\OO(d) - Z)) =0$, if $|Z| < 3d -5$. The first term 
        \begin{equation}
    h^1(\OO(d) - 2 Z)=
    \begin{cases}
        0 & \text{if} \ 2|Z| < 3d \text{or} 2|Z|= 3d \text{and} 2Z \neq \OO(d)\\
      1, & \text{if}\ 2 Z =\OO(d) \\
      2|Z| - 3d, & \text{otherwise}
    \end{cases}
  \end{equation}
    \end{enumerate}
\end{prop}
\begin{proof}
    We  begin with the proof of (1).  By Proposition \ref{h1N}, $h^1(\PP^n, I^2_Z(d)) = h^1(  N(T), I^2_Z(d)) = h^1(T, \OO(d) - 2|Z|) + h^1(T, N^*  \otimes (\OO(d) - |Z|) )$
    Here $N^*$ is the conormal bundle of $T$ in $\PP^n$, and $|Z|$ is the line bundle on $T$ corresponding to $Z$. Since $T$ is isomorphic to $\PP^1$,  the above two $h^1$ terms only depend on $|Z|$ and  it is also easy to see that $$ |h^1(T, \OO(d) - 2|Z|) + h^1(T, N^*  \otimes (\OO(d) - |Z|) ) $$
    $$- \max(2|Z| - rd,0)| + (n-1)\max(|Z|-d,0) \le K$$  for some constant $K$.

    We move on to the proof of (2). We note that since $d\gg e$, the restriction map $$H^0(\PP^n ,\OO(d)) \to H^0(  N(C_1 \cup C_2) , \OO(d))$$ is surjective (by Serre vanishing).  Since $C_1$ and $C_2$ are disjoint, $  N(C_1\cup C_2) =   N(C_1) \cup   N(C_2)$ and also $$H^0(  N(C_1 \cup C_2), \OO(d)) \cong H^0(  N(C_1 ), \OO(d)) \oplus H^0(  N( C_2), \OO(d)).$$ The map $$H^0(   N(C_1 \cup C_2), \OO(d)) \to H^0(  N(Z), \OO(d))$$ is then the direct sum of the maps $$ H^0(N(C_1 ), \OO(d)) \to H^0(N(Z_1), \OO(d))$$ and $$ H^0(  N(C_2 ), \OO(d)) \to H^0(N(Z_2), \OO(d)).$$ We note that  $h^1(\PP^n, I^2_Z(d))$ is the codimension of the image of first map and $h^1(\PP^n, I^2_{Z_1}(d))$ (resp. $h^1(\PP^n, I^2_{Z_2}(d))$) is the codimension of the second (resp. third) map. Since codimension is additive over direct sums, we have the required equality.

    We now move on to the proofs of (3) and (4). Since $d \gg 0$, $h^1(\PP^n, I^2_Z(d)) = h^1(  N(L_1 \cup L_2), I^2_Z(d))$. We note that we have a map $\pi:   N(L_1) \amalg   N(L_2) \to   N(L_1 \cup L_2)$. There is a short exact sequence of sheaves on $  N(L_1 \cup L_2)$, $$0 \to I^2_Z(d) \to \pi_* \pi^* I^2_Z(d) \to \mathcal{F} \to 0,$$ where the quotient sheaf $\mathcal{F}$ is  supported on $N(p)$, where $p$ is the intersection point of the two lines. Associated to the above short exact sequence of sheaves we have a long exact sequence in cohomology, a part of which is as follows:

    $$ H^0(  N(L_1 \cup L_2), \mathcal{F}) \to H^1(  N(L_1 \cup L_2),I^2(d)) \to H^1(  N(L_1 \cup L_2) , \pi_* \pi^* I^2_Z(d)) \to 0.$$
    The fact that the above long exact sequence is exact at the last term is due to the fact that $H^1( N(L_1 \cup L_2), \mathcal{F}) =0$ as its support is $0$ dimensional. Furthermore, $H^0(  N(L_1 \cup L_2), \mathcal{F}) \cong  H^0(N(p), I^2_{Z \cap p}(d))$. Hence $h^0(N(p), I^2_{Z \cap p}(d)) \le  n+1$ (it is either $0$ or $n+1$ depending on whether $p \in Z$).
    
    We note that $$H^1(  N(L_1 \cup L_2) , \pi_* \pi^* I^2_Z(d)) \cong H^1(  N(L_1),  I^2_{Z \cap L_1}(d)) \oplus H^1(  N(L_2),  I^2_{Z \cap L_2}(d)). $$
    
    Let $$M = Im( H^0(  N(L_1 \cup L_2), \mathcal{F}) \to H^1(  N(L_1 \cup L_2),I^2(d))).$$
   Then by the long exact sequence we have that $$h^1(  N(L_1 \cup L_2), I^2_Z(d)) = h^1(  N(L_1), I^2_{Z \cap L_1}(d)) + h^1(  N(L_2), I^2_{Z \cap L_2}(d)) + \dim_\FF M.$$ We note that $\dim_\FF M \le h^0(\mathcal{F}) \le n+1$
   
   We further note that $$0 \le h^1(  N(L_i), I^2_{Z \cap L_i}(d)) - h^1(  N(L_i), I^2_{Z_i}(d)) \le n+1.$$ 
    Hence $$|h^1(\PP^n,I^2_Z(d)) - h^1( N(L_1), I^2_{Z_1}(d)) - h^1(  N(L_2), I^2_{Z_2}(d)) | \le (n+1)$$ in the case when $ p \not \in Z$ (the situation of case (3)) and 
    $$|h^1(\PP^n,I^2_Z(d)) - h^1(  N(L_1), I^2_{Z_1}(d)) - h^1(  N(L_2), I^2_{Z_2}(d)) | \le 3(n+1)$$ in the case when $ p  \in Z$ (the situation of case (4)).

    To finish the proof of (3) and (4) we must show that $\dim M$ only depends on $|Z_1|$ and $ |Z_2|$. Since $\dim M$ is also equal to $$ h^0(N(L_1 \cup L_2), \mathcal{F}) - \dim \textrm{ker}  H^0(  N(L_1 \cup L_2), \mathcal{F}) \to H^1(  N(L_1 \cup L_2),I^2(d))$$ 
    $$ = h^0(N(L_1 \cup L_2), \mathcal{F}) - \dim  Im(H^0(  N(L_1 \cup L_2), \pi_*\pi^* (I^2(Z(d)))) \to H^0(  N(L_1 \cup L_2),  \mathcal{F})),  $$ it suffices to show that $$\dim (Im(H^0(   N(L_1 \cup L_2), \pi_*\pi^* (I^2(Z(d)))) \to H^0(   N(L_1 \cup L_2),  \mathcal{F})) $$ depends only on $|Z_1|$ and $|Z_2|$. In case $p \in Z $, $H^0(  N(L_1 \cup L_2),  \mathcal{F}) =0$, so this is immediate. This establishes (4). We will now complete the proof of (3).

    The map $$H^0(  N(L_1 \cup L_2), \pi_*\pi^* (I^2(Z(d)))) \to H^0(  N(L_1 \cup L_2),  \mathcal{F})$$ is isomorphic to the map 
    $$H^0(  N(L_1), I^2_{Z_1}(d))  \oplus H^0(  N(L_2), I^2_{Z_2}(d)) \to H^0(N(p), \OO(d)) $$ given by restriction. It suffices to show that the image of $H^0(  N(L_1), I^2_{Z_1}(d)) \to H^0(N(p), \OO(d)) $ is dependent on $|Z_1|$ only. This happens in all but one case. We record what happens below without proof since it is elementary and just a computation.

   If $2|Z_1| + 1 \le d$, the map $H^0(  N(L_1), I^2_{Z_1}(d)) \to H^0(N(p), \OO(d)) $ is surjective. We note that if $2|Z_1| \ge  d+1$ and $|Z_1| < d$ the image consists of all vectors vanishing at $ p$ and whose first derivative along the tangent direction to $L_1$ also vanishes at $p$.
   If $|Z_1| \ge d$, $H^0(N(L), I^2_{Z_1}(d)) =0$, so the image is also zero.
   The remaining  case is when $2|Z_1| = d,$ In this situation the image is constrained by the fact that any $f \in H^0(  N(L_1), I^2_{Z_1}(d)) $is such that $(f(p), \partial_v f(p))$ is forced to lie in some fixed one dimensional linear subspace.

   Now the only way that the image of  $$H^0(  N(L_1), I^2_{Z_1}(d))  \oplus H^0(  N(L_2), I^2_{Z_2}(d)) \to H^0(N(p), \OO(d)) $$ can have different dimensions for the same value of $|Z_1|, |Z_2|$ is if either $2|Z_1| =d$ or $2|Z_2| = d$.

   Assume that $2|Z_1| =d$ and $2|Z_2|+ 1 \le d$. Then in this case the map is surjective, so the dimension of the image is fixed. Similarly, if  $2|Z_1| =d$ and $|Z_2|\ge d$, the image is of dimension $n$ and dependent only on the sizes of $Z_1, Z_2$. If $2|Z_1| =d$ and $2|Z_2| \ge d+1$ one gets that the image is still surjective. In the case when $2|Z_1| = 2|Z_2|=d$ the map is still surjective and in particular only  depends on the sizes of $|Z_1|$ and $|Z_2|$.

    The proof of (5)- (10) are almost identical to that of (3) and (4), so we omit them.
    Case (11) is a standard application of Riemann-Roch for an elliptic curve.
    
\end{proof}

\begin{prop}\label{lowdeg}
    Let $M$ be one of the following: 
    \begin{enumerate}
        \item The space of twisted cubic curves in $\textrm{Chow}(3)$.
        \item The space of skew triples of lines in  $\textrm{Chow}(3)$.
        \item The space of  pairs of a line and a disjoint conic in $\textrm{Chow}(3)$.
        \item The space  of pairs of a line intersecting a conic at a point in $\textrm{Chow}(3)$.
        \item The space  of pairs of a line intersecting a conic at two points in $\textrm{Chow}(3)$.
        \item The space of smooth planar conics in $\textrm{Chow}(2)$.
        \item The space of skew lines in $\textrm{Chow}(2)$.
        \item The space of coplanar lines in $\textrm{Chow}(2)$.
        \item $ \textrm{Chow}(1)$.
    \end{enumerate}
Then there exists $\epsilon>0$ $$\sum_{|\lambda| \le N}(-1)^{||\lambda||} W^M_{\ge \lambda} =0,$$ up to dimension $-(\frac{3n}{2} + \epsilon)d + h^0(\PP^n, \OO(d))$.
    Furthermore  for $k \le N$, $$\sum_{|\lambda|= k} (-1)^{||\lambda||} w^{M, l}_{\lambda} =0$$ up to dimension $-(\frac{3n}{2} + \epsilon)d -l +|\lambda|n $.
    .

   Furthermore for $k \le N$, $$\sum_{|\lambda| =k} (-1)^{||\lambda||} w^{M, l,n}_{\lambda} =0,$$ up to dimension  $-(\frac{3n}{2} + \epsilon)d -l +|\lambda|n $.
\end{prop}
\begin{proof}
The proof of the Proposition \ref{lowdeg} in these 9 separate cases are almost identical. However there are some minor differences. We will prove the Proposition \ref{lowdeg} in detail in the cases numbered 1,2 and 4. We will indicate how the proof is to be adjusted in all remaining cases.

Let us begin with the proof in the case when $M$ is the moduli space of twisted cubic curves, i.e. case (1).

 Let $T$ be a fixed twisted cubic curve. Then, $$\bar w_{\lambda}^{M,l} = M w_{\lambda}^{T,l} $$ since the family of twisted cubics over $M$ is Zariski locally-trivial. 
 Note that  $$\bar w_{\lambda}^{T,l} = \pi_{\lambda}^{-1}(\bar w_{|\lambda|}^{M,l} ).$$  By Proposition \ref{h1list} $\bar w_{|\lambda|}^{\{T\},l} $ consists of pairs $(Z,T)$  and $|Z \cap T| = k_1$ (where $k_1$ is some number depending on $l$ necessarily bigger than $3$ for $d$ big enough). Hence $$\bar w_{\lambda} ^{T,l} =\sum_{\alpha \subseteq \lambda, |\alpha| = k_1 } w_{\alpha}(T)w_{\lambda - \alpha}(\PP^n \setminus T)$$

By Proposition \ref{polyinL} $w_{\alpha}(T)w_{\lambda - \alpha}(\PP^n \setminus T)$ is a polynomial in $\Lb$. Thus it suffices to establish that the image of the above class in $K_0(MHS)$ is $0$.  By Proposition \ref{pm} the image of $$ \sum_{|\lambda| = k_1 +k_2} (-1)^{||\lambda||}\sum_{\alpha \subseteq \lambda, |\alpha| = k_1, |\lambda| = k_1 + k_2} w_{\alpha}(T)w_{\lambda - \alpha}(\PP^n \setminus T)$$ in $K_0(MHS)$ is the class of $H_*^{c}(w_{k_1}(T) \times w_{k_2} (\PP^n \setminus T),\pm \QQ)$.
 by Lemma 2  in \cite{V} is $0$, which establishes that 
$$\sum_ {|\lambda| = k_1 + k_2 }\bar w_{\lambda}^{T,l}(\PP^n) =0.$$

By Proposition \ref{barminus} $$\bar w_{\lambda}^{M,l}(\PP^n) -  w_{\lambda}^{M,l}(\PP^n)  \subseteq \cup_{d=1}^{\infty} w_{\lambda}^{M'(d),l}(\PP^n),$$ where $M'(d)$ is the moduli space of degree $d$ curves strictly containing a twisted cubic. By Proposition \ref{ge4}  this is of dimension $$\le  h^0(\PP^n ,\OO(d)) - (\frac{3n}{2} +\epsilon)d - (n+1)|\lambda| -l.$$
Hence $$\sum_ {|\lambda| = k_1 + k_2 } w_{\lambda}^{T,l}(\PP^n) =0,$$ up to dimension $$\le  h^0(\PP^n ,\OO(d)) +(- \frac{3n}{2} - \epsilon)d - (n+1)||\lambda|| -l.$$
Note that in case (1), for a fixed value of $l$, $w_{\lambda}^{M,l,n}$ is either equal to $w_{\lambda}^{M,l}$ or is $0$. Hence $$\sum_ {|\lambda| = k_1 + k_2 } w_{\lambda}^{T,l,n}(\PP^n) =0,$$ up to dimension   $$h^0(\PP^n ,\OO(d)) - (\frac{3n}{2} - \epsilon)d - (n+1)|\lambda| -l.$$  Since $W_{\ge \lambda}^{M} = \sum_{l=0}^{\infty} W_{\ge \lambda}^{M,l}$ and $W_{\ge \lambda}^{M,l} = w_{\ge \lambda}^{M,l} \Lb^{h^0(\PP^n ,\OO(d)) -|\lambda|(n+1) +l}$ case (1) follows.

We will now deal with case when $M$ is the space of skew triples of lines.
Claim: Let $L_1, L_2, L_3$ be a fixed skew triple of lines. Let $Z \subseteq L_1$ be a subset of size $n$. Let $f(n) = h^1(    N(L),I^2_Z(d))$ (One can check that this quantity only depends on $|Z|$ ). Let $S_l = \{(a,b,c)| f(a) + f(b) +f(c) =l\}$. One can check that $S_l$ is  a finite set. 

Firstly, since any family of skew triples of lines is Zariski locally trivial, $$\sum_{|\lambda| = k}(-1)^{||\lambda||}\bar w_{\lambda}^{M,l}(\PP^n) = M \sum_{|\lambda| = k}(-1)^{||\lambda||}\bar w_{\lambda}^{L_1 \cup L_2 \cup L_3,l}(\PP^n) .$$

By Proposition \ref{h1list}, this equals

$$\sum_{|\lambda| = k}(-1)^{||\lambda||} \sum_{(k_1,k_2,k_3) \in S_l} \sum_{ \substack{ \alpha_3\subseteq\alpha_2\subseteq\alpha_1 \subseteq \lambda \\ (|\alpha_1|, |\alpha_2|, |\alpha_3|) = (k_1,k_2,k_3)}} 
\bar w_{\alpha_3}(L_3)\bar w_{\alpha_2 - \alpha_3}( L_2) \bar w_{\alpha_1 - \alpha_2}(L_1) \bar w_{\lambda - \alpha_1}(\PP^n - L_1 \cup L_2 \cup L_3).$$ 

By Proposition \ref{polyinL} the right hand side of the above equation  is a polynomial in $\Lb$.
By Proposition \ref{pm}, the image of the right hand side in $K_0(MHS)$ is a sum of terms of the form $H_*^{c}(w_{k_1}(L_1) \times w_{k_2}(L_2) \times w_{k_3}(L_3)  )$which by Proposition in \cite{V}  is zero. This implies $\sum_{|\lambda| = k}(-1)^{||\lambda||}\bar w_{\lambda}^{M,l}(\PP^n) =0$ . As in case(1) $$\sum_{|\lambda| = k}(-1)^{||\lambda||}\bar w_{\lambda}^{M,l}(\PP^n) - w_{\lambda}^{M,l}(\PP^n)  $$ is of dimension less than equal to that of $\cup_{d=1}^{\infty} w_{\lambda}^{M'(d),l}(\PP^n),$ where $M'(d)$ is the moduli space of degree $d$ curves strictly containing a skew triple of lines. By Proposition \ref{ge4} this is $\le h^0(\PP^n ,\OO(d)) - (\frac{3n}{2} - \epsilon)d - (n+1)||\lambda|| -l.$  Case (4) follows.

The proofs of the other cases are nearly identical. In each case we prove that $\bar w_{\lambda}^{M,l}  = M \bar w_{\lambda}^{\{C\}, l}$  where $C \in M$, since $M$ is Zariski locally trivial in all of the cases.  We then use Proposition \ref{h1list} to explicitly describe $\bar w_{|\lambda|}^{\{C\},l}$. In each of our cases $H^*(\bar w_{|\lambda|}^{\{C\},l}, \pm \QQ) =0$. As a result by Proposition \ref{pm} $\sum_{|\lambda|= n} (-1)^{||\lambda||}\bar w_{\lambda}^{\{C\},l}$
is mapped to $0$ in $K_0(MHS).$

We then note that  $\bar w_{\lambda}^{\{C\},l} = \pi_{\lambda}^{-1}\bar w_{|\lambda|}^{\{C\},l} $ is seen to be a polynomial in $\Lb$. Hence $\sum_{|\lambda|= n} (-1)^{||\lambda||}\bar w_{\lambda}^{\{C\},l} = 0 \in \M$.

We then note that $$\sum_{|\lambda|= n} (-1)^{||\lambda||}(\bar w_{\lambda}^{M,l} - w_{\lambda}^{M,l}) = \sum_l\sum_{M',|\lambda|= n}(-1)^{||\lambda||}w_{\lambda}^{M',l}.$$ where $M'$ ranges over the moduli spaces of curves strictly containing a curve in $M$. However either $M'$ consists of curves of degree $\ge 4$ and hence $w_{\lambda}^{M',l}$ is of dimension $\le h^0(\PP^n ,\OO(d)) - (\frac{3n}{2} - \epsilon)d - (n+1)|\lambda| -l $ by Proposition \ref{ge4}, or $M'$ is a case that has already been considered on this list and hence has dimension $ \le h^0(\PP^n ,\OO(d)) - (\frac{3n}{2} - \epsilon)d - (n+1)|\lambda| -l$.
\end{proof}

\begin{prop}
    Let $M \subseteq Chow(3)$ be contained in the subset parametrizing three lines.
    Then there exists $\epsilon >0$,  such that for $N = \lceil(\frac{3n}{2} + \epsilon)d\rceil$, and $l>0$, $\sum_{|\lambda| < N} (-1)^{||\lambda||}w_{\lambda}^{M,l}$ is of dimension $\le -l + |\lambda|n + -(\frac{3n}{2} + \epsilon)d$
\end{prop}
    \begin{proof}
     We note that we may express $M$ as a disjoint union of constructible sets $M = \amalg_{i=1}^k M_i$ such that for each $i$, and for each triple of lines $\{L_1 ,L_2, L_3\}$ in $M_i$ the triple has the same configuration, i.e. the number of points of intersection of the lines remain constant in $L_i$. It suffices to prove the Proposition for each $M_i$ separately. Thus in what follows we assume that each triple of lines in $M$ has the same configuration and $M$ is irreducible.   We note that under this assumption the family of triples of lines over $M$ is Zariski locally trivial. 
    
       By Proposition \ref{barminus}  $$\bar w_{\lambda}^{M,l} - w_{\lambda}^{M,l} \subseteq \cup_{d'\ge 4,l'>l} w_{\lambda}^{Chow(d'),l'} $$ and by Proposition \ref{ge4} the right hand side of the above expression is of dimension $\le -l + |\lambda|n + -(\frac{3n}{2} + \epsilon)d$. Therefore it suffices to prove that $$\sum_{|\lambda|<N} (-1)^{||\lambda||}\bar w_{\lambda}^{M,l}$$ is of dimension $\le -l + |\lambda|n + -(\frac{3n}{2} + \epsilon)d$.

        Let $A = \ZZ^7 / S_3$ where $S_3$ denotes the symmetric group of three letters and the action of $S_3$ on $\ZZ^7$ is induced from the natural action of $S_3$ on the power set of $\{1,2,3\} \setminus \varnothing$. A configuration of points $Z \subseteq \PP^n$ and a set of 3 lines $\{L_1, L_2, L_3\}$ naturally gives us  $|Z \cap \{L_1, L_2, L_3\}| := (|Z \cap L_1|,|Z \cap L_2|, |Z \cap L_3|, |Z \cap L_1 \cap L_2|, \dots |Z \cap L_1 \cap L_2 \cap L_3|)$ which is a well defined element of $A$. 
        
        Given $a \in A$, let $$\bar w_{\lambda}^{M}(a) = \{(Z, \{L_1, L_2, L_3\}) \in w_{\lambda}^M| |Z \cap \{L_1, L_2, L_3\}|  =a\}.$$
        We may now further stratify each $\bar w_{\lambda}^{M,l}$ as the disjoint  union of subspaces $\bar w_{\lambda}^{M, l} \cap \bar w_{\lambda}^{M}(a) $ where $a \in A$. 

       However by Proposition \ref{h1list}, $\bar w_{\lambda}^{M,l}(a)$ is either all of $\bar w_{\lambda}^{M}(a)$ or it is empty. We note that since the family of triples of lines over $M$ is Zariski locally trivial, $\bar w_{\lambda}^{M}(a)$ is also a locally trivial family over $M$. 
       
       Thus since Zariski locally trivial families split into prosucts in $\M$, we have the following equality: $$[\bar w_{\lambda}^{M}(a)] = [M ][w_{\lambda}^{c}(a)] , $$ where $c$ is a point in $M$. Thus it suffices to establish that for a fixed  $c \in M$ and $a \in A$ $\sum_{|\lambda| < N} (-1)^{||\lambda||}w_{\lambda}^{c}(a) =0$.
       
       Since the above expression is polynomial in $\Lb$, it suffices to prove that its image  in $K_0(MHS)$ is $0$. By Proposition \ref{pm}  $\sum_{|\lambda| = r} (-1)^{||\lambda||}w_{\lambda}^{c}(a)$ is mapped to $H_*^{c}(w_{[r]}^{c} (a))$. The topological space $w_{[r]}^{c} (a)$  has a factor of the form  $w_{[r]}^{c}(\A^1)$ . However by Proposition \ref{pm} $H_*^c(w_{[r]}^{c}(\A^1),\pm \QQ) =0$ and thus $H_*^{c}(w_{[r]}^{c} (a)) =0$. This concludes the proof.
    \end{proof}

\begin{defn}
    Let $\epsilon> 0$. Let $M$ denote the moduli space of smooth plane cubics in $\PP^n$. Let $N = (\frac{3n}{2} + \epsilon)d$. Let $$Y(d):= \sum_{l >0}\sum_{|\lambda|< N} ((-1)^{||\lambda||}) w_{\lambda}^{M,l} \Lb^{n|\lambda|}(\Lb^l -1).$$

\end{defn}
We note that a priori this definition of $Y(d)$ depends on our choice of $\epsilon> 0$. However if $\epsilon$ is sufficiently small, different values of $\epsilon$ will only change the value of $Y(d)$ in very high codimension, and this is an ambiguity that is acceptable for our purposes.

\begin{prop}\label{Yd}
   There exists $\epsilon >0$ such that:  Up to dimension  $h^0(\PP^n \OO(d))-(\frac{3n}{2} + \epsilon)d$, $$Y(d)= \sum_{l >0}\sum_{\lambda} ((-1)^{||\lambda||}) \bar w_{\lambda}^{M,l} \Lb^{n|\lambda|}(\Lb^l -1).$$ 
   Let us now assume $d$ is even.
   Let $k_1(l) = \frac{3d+l}{2}$.

   Then  up to dimension  $h^0(\PP^n \OO(d))-(\frac{3n}{2} + \epsilon)d$, $$\sum_{l >0}\sum_{\lambda} ((-1)^{||\lambda||}) \bar w_{\lambda}^{M,l} \Lb^{n|\lambda|}(\Lb^l -1)$$

    $$=\sum_{l >0, l \textrm{ odd}}\sum_{|\lambda| < N} \sum_{\alpha\subseteq \lambda, |\alpha| = k_1(l)}((-1)^{||\lambda||})\{(Z_1,Z_2,C)| C \in M, Z_1 \in w_{\alpha}(C), Z_2 \in w_{\lambda \setminus \alpha} (\PP^n \setminus C)\}  \Lb^{-(n+1)|\lambda|}(\Lb^l -1).$$ 

\end{prop}
\begin{proof}
    Let $M$ denote the space of plane cubics in $\PP^n$.
    We note that $$Y(d)- \sum_{l >0}\sum_{\lambda} ((-1)^{||\lambda||}) \bar w_{\lambda}^{M,l} \Lb^{n|\lambda|}(\Lb^l -1),$$ equals $$\sum_{l >0}\sum_{\lambda} ((-1)^{||\lambda||}) w_{\lambda}^{M ,l} -\bar w_{\lambda}^{M,l} \Lb^{n|\lambda|}(\Lb^l -1).$$

     By Proposition \ref{ge4} and Proposition \ref{barminus} each difference $w_{\lambda}^{M ,l} -\bar w_{\lambda}^{M,l} \Lb^{n|\lambda|}(\Lb^l -1)$ is of dimension $\le h^0(\PP^n \OO(d))-(\frac{3n}{2} + \epsilon)d$. Thus the entire sum is of dimension $\le  h^0(\PP^n \OO(d))-(\frac{3n}{2} + \epsilon)d.$

     If $d$ is even and $l<\frac{3d}{2}$ we have $$ bar w_{\lambda}^{M,l} = \sum_{\alpha\subseteq \lambda, |\alpha| = k_1(l)}((-1)^{||\lambda||})\{(Z_1,Z_2,C)| C \in M, Z_1 \in w_{\alpha}(C), Z_2 \in w_{\lambda \setminus \alpha} (\PP^n \setminus C)\},$$ by  Proposition $\ref{h1list} (11).$ To see this note that under the assumption $l < \frac{3d}{2}$, it must be the case that $h^1(N(C),I^2_{Z \cap C}(d)) = h^1(C, \OO(d) -2|Z\cap C|)$. Since $C$ is of genus $1$, $ h^1(C, \OO(d) -2|Z\cap C|) = h^0(C, 2|Z\cap C| - \OO(d)) = 2|Z \cap C| -3d$ (by our parity assumption $3d \neq 2|Z \cap C|$ and  the case when $2|Z \cap C| - \OO(d)$ is the trivial line bundle does not arise).

     We note that we may disregard the terms when $l \ge \frac{3d}{2}$ because they contribute in too high a codimension.

    Thus  we have our desired equality.
\end{proof}
\begin{prop}\label{cubics}
    Let $M$ be the space of plane cubics in $\PP^n$. Let $d \equiv 1 \mod{4}$. Let $k_0 = \frac{3d+1}{2}$. Then for $\epsilon$ sufficiently small (the definition of $Y(d)$ implicitly depends on an $\epsilon$), $Y(d)$ has the same highest weight term as  $$ H^*(\M_{1,1}, H_{k_0}) H^*(PGL_{n+1}(\CC))\Lb^{-k_0(n+1)+ 1}$$  in $K_0(MHS)$. The weight of this highest weight term is $2(- k_0(n+1) +1) +k_0 + 1 + (n^2 -1)$.
\end{prop}
\begin{proof}
    By Proposition \ref{Yd}, up to dimension $d(\frac{3n}{ 2} + \epsilon)$ $Y(d)  $ 
    $$=\sum_{l >0, l \textrm{ odd}}\sum_{|\lambda| < N} \sum_{\alpha\subseteq \lambda, |\alpha| = k_1(l)}((-1)^{||\lambda||})\{(Z_1,Z_2,C)| C \in M, Z_1 \in w_{\alpha}(C), Z_2 \in w_{\lambda \setminus \alpha} (\PP^n \setminus C)\}  \Lb^{-(n+1)|\lambda|}(\Lb^l -1).$$

    To finish the proof of the Proposition it suffices to compute the image  of  the above expression in $K_0(MHS)$ . 
    We note that the image equals $\sum_{l,k}H^*(w_{[k]}^{M,l}, \pm \QQ) \Lb^{-n|\lambda| + l} $. 

    For $l=1$, $\sum_{k}H^*(w_{[k]}^{M,1}, \pm \QQ) \Lb^{-nk + l}$ has highest weight term equal to that of $H^*(w_{k_0}^{M,1}, \pm \QQ)$, where $k_0 = \frac{3d+1}{2}$, there are other terms corresponding to higher values of $k$, but we will establish (after first dealing with the leading term)that these have lower weight. This highest weight term corresponds to the situation where we have $k_0$ points on a plane cubic $C$, in this situation being singular on such a collection of points imposes $(k_0)(n+1) -1$ linear conditions.

    Now $w_{k_0}^{M,1} = \{(Z,C)\in w_{k_0}\times M| Z \subseteq C \}$(by Proposition \ref{h1list}) and hence $w_{k_0}^{M,1} / PGL_{n+1}(\CC) = \M_{1,k}$. So  we have,
    
 $H^*(w_{k_0}^{M,1}, \pm \QQ) = H^*(PGL_{n+1}, \QQ) \sum_{l,n}H^*(\M_{1,k_0}, \pm \QQ) $. By  the results of Section 5 of \cite{G} (also see Corollary 2.8 in \cite{G2} )      the highest weight term of this equals that of $$H^*(PGL_{n+1}(\CC), \QQ)H^*(\M_{1,1}, H_{k_0}),$$ we note that this then implies that $H^*(w_{[k_0]}^{M,1}, \pm \QQ) \Lb^{-nk_0 + 1}$ has the required highest weight, see \cite{G} for instance.

We will now establish that for $k> k_0$, the term $H^*(w_{[k]}^{M,1}, \pm \QQ) \Lb^{-(n+1)k + l}$ has lower weight than for $k =k_0$. To see this let $\bar w_{[k]}^{M ,1} = \{(Z,C) \in w_{[k]} | h^1(\PP^n ,I^2 _{Z \cap C}(d)) =1\}$. We note that $\bar w_{[k]}^{M ,1} \setminus w_{[k]}^{M ,1}$ is of small dimension by Proposition \ref{ge4}. Thus it suffices to establish that  $H^*( \bar w_{[k]}^{M,1}, \pm \QQ) \Lb^{-nk + l}$ has lower weight. But we note that by Proposition \ref{h1list} $\bar w_{[k]}^{M,1} = \{(Z,C) \in w_{[k]} \times C | |Z \cap C| = k_0\}.$ Now we have a fibration $w_{[k]}^{M,1} \to M $ with fiber isomorphic to $w_{[k_0]}(C) \times w_{k-k_0}(\PP^n \setminus C)$. We have a further fibration (in the orbifold sense) $M \to M / PGL_{n+1}(\CC).$ The fiber of the composite map $w_{[k]}^{M,1} \to M / PGL_{n+1}(\CC)$ is seen to be a copy of  $PGL_{n+1}(\CC) \times w_{[k_0]}(C) \times w_{k-k_0}(\PP^n \setminus C).$

We may thus use the Leray spectral sequence to bound the weights of $H_*^c(w_{[k]}^{M,1}), \pm \QQ)$ in terms of the dimension of $M/PGL_{n+1}(\CC)$ (which is complex 1 dimensional and non compact) and the weights of $H_*^cPGL_{n+1}(\CC) \times w_{[k_0]}(C) \times w_{k-k_0}(\PP^n \setminus C))$. The highest weight term of the latter is $((n+1)^2 - 1 ) + k_0 +1 + (2n-1)(k-k_0) +1$ (this uses a computation of the Borel-Moore homology of configuration spaces- see Corollary 7.2.6 of \cite{DH} for instance).

Thus the highest weight term of $H_*^c(w_{[k]}^{M,1}), \pm \QQ)$ is at most $1  +((n+1)^2 -1) + k_0 + 1 + (2n-1)(k-k_0)$.
Thus the highest weight of $$H_*^c(w_{[k]}^{M,1}), \pm \QQ)\Lb^{-k(n+1) +1}$$ is at most $1  +((n+1)^2 -1) + k_0 + 1 + (2n-1)(k-k_0) + 2 - 2k(n+1)$. If we subtract this quantity from $2(- k_0(n+1) +1) +k_0 + 1 + ((n+1)^2 -1)$  we obtain:
$$3(k-k_0)- 1 $$ which is always positive for $k>k_0$. Thus the possible weights are lower than the highest term.

 We note that for $l>1$ a similar computation shows that all the corresponding terms  $H^*(w_{[k]}^{M,l})\Lb^{-nk+ l}$ are of smaller weight and we may ignore them for the purpose of finding the highest weight term.

 This establishes the result.

\end{proof}

\begin{prop} \label{geN0}
    Let $d \gg n$. Let $L$ be a fixed line in $\PP^n$. Then there exists, $\epsilon > 0$  such that   for $N= (\frac{3n}{2} + \epsilon)d$, we have the following inequalities/equalities.
    \begin{enumerate}
        \item$$\{(f,Z) \in  W_{\lambda, \ge N}^{l,n}|  \dim \Sing(f)= 0 \} $$ is of dimension $\le h^0(\PP^n,\OO(d)) +  (- \frac{3n}{2} -\epsilon)d$
        \item $W_{\lambda, \ge N}^{0,n}$is of dimension $\le h^0(\PP^n,\OO(d)) +  (- \frac{3n}{2} -\epsilon)d$
     
        \item $$\{(f,Z) \in  W_{\lambda, \ge N}^{l,n}|  \Sing(f) \textrm{contains a curve of degree} \ge 3\} $$ is of dimension $\le h^0(\PP^n,\OO(d)) +  (- \frac{3n}{2} -\epsilon)d$.
        \item $$W_{\lambda, \ge N}^{l,n} = W_{\lambda, \ge N}^{l,n} \cap \{(f,Z) | C \subseteq \Sing (f), C \textrm{ is a curve of degree } 
        \le 2\}$$
        $$ = \{(f,L,Z)| (f,Z)\in W_{\lambda, \ge N}^{l,n}, L \subseteq \Sing(f), L \textrm{ is a line} \}$$
        $$+ \{(f,C,Z)| (f,Z)\in W_{\lambda, \ge N}^{l,n}, C \subseteq \Sing(f) , C \textrm{is an irreducible conic}\}$$
        
        $$- \{(f,L_1,L_2,Z)| (f,Z)\in W_{\lambda, \ge N}^{l,n}, L_i \subseteq \Sing(f) 
 L_i \textrm {are distinct lines}\},$$  up to dimension $\le h^0(\PP^n,\OO(d)) +  (- \frac{3n}{2} -\epsilon)d.$
       
        \item$\sum_{|\lambda|< N} (-1)^{||\lambda||}  \{(f,L_1,L_2,Z)| (f,Z)\in W_{\lambda, \ge N}^{l,n}, L_i \subseteq \Sing(f) \} = 0$ up to dimension $\le h^0(\PP^n,\OO(d)) +  (- \frac{3n}{2} -\epsilon)d$. 
        \item $\sum_{|\lambda| < N} (-1)^{||\lambda||}  \{(f,C,Z)| (f,Z)\in W_{\lambda, \ge N}^{l,n}, C \subseteq \Sing(f) \} = 0$ up to dimension $\le h^0(\PP^n,\OO(d)) +  (- \frac{3n}{2} -\epsilon)d$. 
        \item $\sum_{|\lambda|< N} (-1)^{||\lambda||}  \{(f,L,Z)| (f,Z)\in W_{\lambda, \ge N}^{l,n}, L \subseteq \Sing(f) 
 L \textrm{ is a line}\} = 0$ up to dimension $\le h^0(\PP^n,\OO(d)) +  (- \frac{3n}{2} -\epsilon)d$. 

    \end{enumerate}

\end{prop}
\begin{proof}
We begin by proving (1). It suffices to prove (1) in the case when $|\lambda| =N$.
By Proposition \ref{contcurve} if $l > 0$, then associated to any  $Z \in W_{\ge \lambda}^{n,l}$ is a unique reduced curve $C$ such that $l = h^1(\PP^n, I^2_Z(d)) = h^1(\PP^n, I^2_{Z \cap C}(d))$. Furthermore, if the components of $C$ are $C_1, \dots C_k$ then $\frac{\deg(C_k)d - g(C_k) }{2} \le |Z \cap C_k| $   as otherwise $2|Z| \cap C_k$ imposes linearly independent conditions which is not possible (see Proposition \ref{contcurve}).  Also, there is some constant $K$ depending on $\deg(C_k)$ such that if  $ |Z \cap C_k| > \deg(C_k) d + K  ,$  a section singular at $Z$ is singular on all of $C_k$. if that happens $Z \not \in W_{\ge \lambda}^{n,l}$.  Thus we must have $  |Z \cap C_k| < \deg(C_k) d + K$.

Furthermore there is some number $K'$ such that , $|l - (n-1)\sum_k (|Z \cap C_k|) | \le K'$(Refer to Proposition \ref{h1bound}). We can therefore bound both the dimension of the space of such $Z$ and $l$, which gives us (after a computation):

, $$\dim \{(f,Z)|  \dim \Sing(f)= 0 \} \cap W_{\lambda, \ge N +1}^{l,n} \le h^0(\PP^n, \OO(d))+( - \frac{3n}{2} - \epsilon )d.$$

Part (2) is immediate, since $$ \dim W_{\lambda, \ge N +1}^{0,n} \le \dim W_{\ge [N]}^{0,n}  \le  h^0(\PP^n, \OO(d))- N .$$ 

Let us now proceed to prove  part (3). 
We first begin with a claim:

\textbf{Claim}: For any fixed curve $C$ of degree $\ge 3$ and a fixed curve $C'$ of degree $r'$ possibly empty, there exists $\epsilon'>0$, such that $$\dim \{ (f,Z) \in W_{\ge \lambda}^{l,n} | C \subseteq \Sing(f),  h^1(\PP^n, I^2_Z(d)) = h^1(\PP^n, I^2_{Z \cap (C\cup C')}(d))\} \le h^0(\PP^n, \OO(d)) - (\frac{3n}{2} + \epsilon')d$$, here we require $\epsilon'$ to be independent of $C', r'$.

To prove the claim, we let $$A =\{(f,Z)|  \Sing(f) \textrm{contains a curve of degree} \ge 3\} \cap W_{\lambda, \ge N}^{l,n} $$ and $\tilde A = \{(f,Z, C)|(f, Z) \in A , C \textrm{ is a curve of degree} \ge 3, C\subseteq \Sing(f)\}.$ We note that $\dim \tilde A \ge \dim A$.  Then $$\dim \tilde A \le \sup_{r \ge 3,C \textrm{is a curve of degree } r} (\dim Chow(r) + \dim \{(f,Z)| C \subseteq \Sing(f) \} \cap W_{\lambda, \ge N}^{l,n}).$$ 
Let us first argue that the claim implies (4). We will establish the claim afterwards.
Let $B_{r'}(C) $ be the subset of $W_{\ge \lambda}^{l,n} \times Chow (r')$ consisting of all $(f,Z, C')$ such that  $C \subseteq \Sing(f) $ and $C'$ is the curve of minimal degree such that $h^1(\PP^n, I^2_Z(d)) = h^1(\PP^n, I^2_{Z \cap (C\cup C')}(d)) \}.$

Now we use Proposition \ref{contcurve} to argue that $$\dim \{(f,Z)| C \subseteq \Sing(f) \} \cap W_{\lambda, \ge N +1}^{l,n}) \le \sup_{r'} (\dim Chow(r') + \dim B_{r'}(C).$$ Note that in the above two expressions, the supremums are actually maximums for a given $|\lambda|$, there are only finitely many $r,r'$ for which the sets $B_{r'}(C)$ can be nonempty. Further, we note that for $d \gg 0$, $\epsilon/ 10 d \gg \dim Chow(r)$ for $r$ in the relevant range. Thus if we establish that $\dim B_{r'}(C) \le  h^0(\PP^n, \OO(d)) - (\frac{3n}{2} + \epsilon)d$ we would be done.

This concludes the proof of (4) assuming the claim, so we will now proceed to establish  the claim. 

For this we note the following- for $d \gg \deg(C)$ the space $\{f \in H^0(\PP^n,\OO(d)) | C \subseteq \Sing(f) \}$ is of codimension $h^0(  N(C), \OO(d))$ in $H^0(\PP^n,\OO(d))$. We know that $h^0(  N(C), \OO(d)) \approx n\deg C d $ by Proposition \ref{h1bound}. For $(f,Z) \in  W_{\ge \lambda}^{C \cup C',l,n}$, let $Z_1 = Z \cap C$,$Z_2 = Z \cap C' \setminus Z_1$ $Z_3 = Z \setminus (Z_1 \cup Z_2)$.
We now have the following inequalities:
\begin{enumerate}
    \item $ \deg(C) d \ge|Z_1|   $.
    \item 
    We note that $|2|Z_2| - \deg(C')d - h^0(\PP^n, I^2_{Z_2}(d))|$ is bounded by a constant independent of $d$.
    \item $\{f| Z \cup C \subseteq \Sing(f)\}$ is of dimension  $h^0(\PP^n, \OO(d)) - h^0(  N(C), \OO(d) ) - h^1 (\PP^n, I^2_{Z'}(d)) -(n+1)|Z_3| - (n+1)|Z_2|$ in .
    \item  But the dimension of the space of $Z$ such that $(f,Z) \in W_{\ge \lambda}^{C \cup C',l,n}$ is less than $|Z_1| + |Z_2| + n|Z_3|$.
    \item  Hence the dimension of$$\{ (f,Z) \in W_{\ge \lambda}^{C \cup C',l,n} | C \subseteq \Sing(f)\} $$
    $$\le  |Z_1| + |Z_2| + n|Z_3| + h^0(\PP^n, \OO(d)) - h^0(  N(C), \OO(d) ) + h^1 (\PP^n, I^2_{Z_2}(d)) -(n+1)|Z_3| - (n+1)|Z_2|
    $$
    $$
     \le h^0(\PP^n ,\OO(d)) + \deg(C) d -n |Z_2|- h^0(N(C), \OO(d)) + h^1(\PP^n ,I^2_{Z_2}(d))
    $$
    $$
    \approx h^0(\PP^n ,\OO(d)) + (n-1)\deg(C) d - n|Z_2| + h^1(\PP^n, I^2_{Z_2}(d)) 
    $$
    $$
    \le  h^0(\PP^n ,\OO(d)) + (n-1)\deg(C) d 
    $$
    
    $$\le h^0(\PP^n, \OO(d)) - (\frac{3n}{2} + \epsilon)d .$$
\end{enumerate}

For part (4), we note that by (1), $$W_{\lambda, \ge N+1}^{l,n} =  X$$ up to dimension $ h^0(\PP^n, \OO(d)) - (\frac{3n}{2} + \epsilon)d ,$ where 
$$X := \{(f,Z) \in  W_{\lambda, \ge N +1}^{l,n} |  \Sing(f) \textrm{contains a curve of degree} \le 2 \textrm{and does not contain a curve of degree} \ge 3\}. $$ 
Standard inclusion-exclusion techniques then  imply that the motive of $$X =\{(f,L,Z)| (f,Z)\in W_{\lambda, \ge N + 1}^{l,n}, L \subseteq \Sing(f) \} + \{(f,C,Z)| (f,Z)\in W_{\lambda, \ge N +1}^{l,n}, C \subseteq \Sing(f) \} $$
$$- \{(f,L_1,L_2,Z)| (f,Z)\in W_{\lambda, \ge N +1}^{l,n}, L_i \subseteq \Sing(f) \}.$$ up to dimension$ h^0(\PP^n, \OO(d)) - (\frac{3n}{2} + \epsilon)d $.



We will omit the proof of parts (5) and (6) because their proofs are similar to that of part (7) (which we will not omit) and in fact a bit simpler.
For part (7), we note that the family of lines over $Chow(1)$ is Zariski locally trivial. Hence $\{(f,L,Z)| (f,Z)\in W_{\lambda, \ge N + 1}^{l,n}, L \subseteq \Sing(f) \}$ defines a locally trivial family over $Chow(1)$ and thus,

$$\{(f,L,Z)| (f,Z)\in W_{\lambda, \ge N + 1}^{l,n}, L \subseteq \Sing(f) \} = Chow(1)\{(f,Z)| (f,Z)\in W_{\lambda, \ge N +1}^{l,n}, L \subseteq \Sing(f) \} $$ in $\M$ (Zariski locally trivial families split into products in $\M$).

We now make the following claim.

\textbf{Claim:} $$ \sum_{|\lambda|= k} (-1)^{||\lambda||}\{(f,Z)| (f,Z)\in W_{\lambda, \ge N + 1}^{l,n}, L \subseteq \Sing(f) \}$$ 
$$ =  \sum_{|\lambda|= k} (-1)^{||\lambda||} \sum_{\alpha \subseteq \lambda} w_{\alpha}(L)w_{\lambda \setminus \alpha} (\PP^n \setminus L) \Lb^{-(n+1)(|\lambda|- |\alpha| - h^0(N(L), \OO(d))}$$ up to dimension $h^0(\PP^n, \OO(d))- (\frac{3n}{2}+ \epsilon )d$, i.e. the sum is well approximated by the  corresponding sum when we ignore the fact that points in $\PP^n \setminus L$ need not impose linearly independent conditions on $H^0(\PP^n, I^2_L(d))$. 

The proof of the claim is essentially a dimension computation similar to those we've already seen, if a collection of points $Z$ fails to impose linearly independent conditions on $H^0(\PP^n, I^2_L(d))$, then a large number of points of $Z$ lie on a curve, either the curve is  of degree $\ge 2$ and the dimension of the resulting error is small (see Proposition \ref{ge4}) or the curve is a line and the alternating sum vanishes for the same reason as in Proposition \ref{lowdeg}. We will first establish (7) assuming the claim.

To do so we must simply establish that $$\sum_{|\lambda|< N} (-1)^{||\lambda||} \sum_{\alpha \subseteq \lambda} w_{\alpha}(L)w_{\lambda \setminus \alpha} (\PP^n \setminus L) \Lb^{-(n+1)(|\lambda|- |\alpha| - h^0(N(L), \OO(d))} = 0$$ up to dimension $h^0(\PP^n, \OO(d))- (\frac{3n}{2}+ \epsilon )d.$ It is clear that the sum is a polynomial in $\Lb$ and thus it suffices to establish that its image in $K_0(MHS)$ is of sufficiently low weight.
Furthermore by Proposition \ref{pm} the image of the above expression in $K_0(MHS)$ is 
$$ \Lb^{-h^0(N(L), \OO(d))}\sum_{k < N}\sum_{k' \le k_1} H_*^c(w_{k_1}(\PP^1)\times w_{k-k_1}(\PP^n \setminus \PP^1), \pm \QQ)\Lb^{-(k-k_1)(n+1)}.$$

To proceed with the proof of (7) we will need the following identity:
$$
\sum_{k= 0}^{\infty} H_*^c(w_k(\PP^1), \pm \QQ) =0.
$$
To establish this we note that $\sum_{k= 0}^{\infty} H_*^c(w_k(\PP^1), \pm \QQ) $ is the image of $\sum_{\lambda} (-1)^{||\lambda||}w_{\lambda}(\PP^1) = Z_{\PP^1}^{-1}(0) =0$. Thus we have established the previous identity. 

We also note that in the sum $\sum_{k= 0}^{\infty} H_*^c(w_k(\PP^1), \pm \QQ) $  the only relevant terms are for $k =0,1,2$ as all other terms vanish by Lemma 1.2 of \cite{V}.

We now note that 

$$\Lb^{-h^0(N(L), \OO(d))}\sum_{k < N}\sum_{k' \le k} H_*^c(w_{k'}(\PP^1)\times w_{k- k'}(\PP^n \setminus \PP^1), \pm \QQ)\Lb^{-(k-k')(n+1)}$$

$$
= \Lb^{-h^0(N(L), \OO(d))} \sum_{k' =0}^2 \sum_{k = k'} ^N 
H_*^c(w_{k'}(\PP^1)\times w_{k- k'}(\PP^n \setminus \PP^1), \pm \QQ)\Lb^{-(k-k')(n+1)}
$$
$$
= \Lb^{-h^0(N(L), \OO(d))} \sum_{k' =0}^2 H_*^c(w_{k'}(\PP^1), \pm \QQ)\sum_{k'' = 0} ^{N -k'-1}
H_*^c( w_{k''}(\PP^n \setminus \PP^1), \pm \QQ)\Lb^{-(k'')(n+1)}
$$

$$
= \Lb^{-h^0(N(L), \OO(d))} \sum_{k' =0}^2 H_*^c(w_{k'}(\PP^1), \pm \QQ)\sum_{k'' = 0} ^{N -1}
H_*^c( w_{k''}(\PP^n \setminus \PP^1), \pm \QQ)\Lb^{-(k'')(n+1)}
$$
$$- \Lb^{-h^0(N(L), \OO(d))} \sum_{k' =0}^2 H_*^c(w_{k'}(\PP^1), \pm \QQ)\sum_{k'' = N -k'-1} ^{N-1}
H_*^c( w_{k''}(\PP^n \setminus \PP^1), \pm \QQ)\Lb^{-(k'')(n+1)}
.$$

By our previous computation the first term is 0. The second term is easily seen to be of dimension $\le -(\frac{3n}{2} + \epsilon )d$. This establishes the proof of (7) assuming the claim.
We will now prove the claim.

Let
$$X_{\lambda,j,l} = \{(f,Z)| (f,Z)\in W_{\lambda, \ge N}^{l,n}, L \subseteq \Sing(f) , |Z \cap L| =j\}$$

Let $Y_{\lambda, j,l} = \sum_{|\alpha| = j , \alpha \subseteq \lambda} w_{\alpha}(L)w_{\lambda \setminus \alpha} (\PP^n \setminus L) \Lb^{-(n+1)(|\lambda|- |\alpha|) - h^0(N(L), \OO(d))}$

We will establish that $\sum_{|\lambda| < N} (-1)^{||\lambda||}X_{\lambda,j,l} - Y_{\lambda,j,l}$ is of dimension $\le h^0(\PP^n ,\OO(d)) -N$ which will immediately give us the claim. 

We may further stratify $X_{\lambda,j,l}$ into the pieces $$X_{\lambda,j,l,k} = \{(f,Z) \in X_{\lambda,j,l} |  h^1(\PP^n, I^2_{Z \cup L}(d)) = h^0(N(L), \OO(d))+ (n+1)(|\lambda| -j) +k\},$$ i.e. $X_{\lambda,j,l,k}$ consists of those $Z$ such that vanishing to order 2 on $Z \cup L$ imposes $k$ fewer linear conditions than expected. Let $Y_{\lambda,j,l,k} = X_{\lambda,j,l,k} \LL^-k$. We note that $\sum_k Y_{\lambda,j,l,k} = Y_{\lambda,j,l}$.

We now note that it suffices to establish that  $\sum_{|\lambda|< N} (-1)^{||\lambda||}X_{\lambda,j,l,k}- Y_{\lambda,j,l,k}$ is of dimension $\le h^0(\PP^n ,\OO(d)) - N$. For $k =0$, the above is exactly $0$.

Let us assume $k >0$. We then  note that by an argument similar to that of Proposition \ref{contcurve} there must exist a curve $C'$ such that $h^1(\PP^n, I^2_{Z \cup L}(d))   =h^1( N(C' \cup L), I^2_{Z \cap (C) \cup L}(d)).$ Further more the above curve $C'$ is uniquely determined by $Z$ assuming it is of minimal degree.

Let $$X_{\lambda,j,l,k}^r \subseteq X_{\lambda,j,l,k} \times Chow(r) $$
 consist of all $(f,Z,C)$ such that $$h^1(\PP^n, I^2_{Z \cup L}(d))   =h^1( N(C \cup L), I^2_{Z \cap (C) \cup L}(d)),$$ and $C$ is of minimal degree.
Thus we have that $X_{\lambda,j,l,k} = \sum_r X_{\lambda,j,l,k}^r$

Let $Y_{\lambda,j,l,k}^r= X_{\lambda,j,l,k}^r\Lb^{-k}$. Note that $\sum_{k,r} Y_{\lambda,j,l,k}^r  = Y_{\lambda,j,l}.$ 

We now consider the difference $X_{\lambda,j,l,k}^r- Y_{\lambda,j,l,k}^r$. It suffices to establish for all fixed values of $j,l,k,r$ that $\sum_{|\lambda|< N} (-1)^{||\lambda||}X_{\lambda,j,l,k}^r- Y_{\lambda,j,l,k}^r$ is of dimension $\le h^0(\PP^n ,\OO(d)) - N$.

It is an easy computation to note that for $ r \ge 2$ the above difference is of dimension $\le h^0(\PP^n ,\OO(d)) -N$. Also for $r =0$, the above term always vanishes. Thus the only  case that we must consider is when $r =1$. Let $M_1$ denote the space of all lines in $\PP^n$ intersecting $L$ at a single point. Let $M_2$ denote the space of all lines in $\PP^n$ not intersecting $L$. Let $$X_{\lambda,j,l,k}^{M_i} = \{(f,Z ,L') \in X_{\lambda,j,l,k}^1, L' \in M_i\}.$$ Let $$Y_{\lambda,j,l,k}^{M_i} = X_{\lambda,j,l,k}^{M_i}\Lb^{-k}.$$

Let $$\bar X_{\lambda,j,l,k}^{M_i}= \{(f,Z,L') \in X_{\lambda,j}|h^1( L' \cup L, I^2_{Z \cap (L') \cup L}(d)) =k \}.$$ Let $$\bar Y_{\lambda,j,k}^{M_i} = \bar X_{\lambda,j,k}^{M_i} \Lb^{-k}.$$

We claim that:

$\bar X_{\lambda,j,l,k}^{M_i} -  X_{\lambda,j,l,k}^{M_i}$ and $\bar Y_{\lambda,j,l,k}^{M_i} -  Y_{\lambda,j,l,k}^{M_i}$ is of dimension $\le h^0(\PP^n, \OO(d)) - N.$  This follows from the fact that these differences are contained in the union of $ X_{\lambda,j,l,k}^{r}$ for $r \ge 2$ and we have already remarked that these have dimension $\le h^0(\PP^n, \OO(d)) - N. $.

Thus it suffices to establish that $$\sum (-1)^{||\lambda||} \bar X_{\lambda,j,l,k}^{M_i} = \sum (-1)^{||\lambda||} bar Y_{\lambda,j,l,k}^{M_i} =0.$$

To do so we will explicitly describe the sets $\bar X_{\lambda,j,l,k}^{M_1}$ and $\bar X_{\lambda,j,l,k}^{M_2}$.

We note that we have a Zariski locally trivial bundle $\bar X_{\lambda,j,l,k}^{M_2} \to Chow(1)$ defined by $(f,Z, L') \mapsto L'$. Let us denote the fibre of this map by $F_{\lambda,j,l,k}^{M_2}$, it is clear that $X_{\lambda,j,l,k}^{M_2} = Chow(1) F_{\lambda,j,l,k}^{M_2}$ The fibre over $L'$ consist of those $Z$ such that  $|Z \cap L| =j$ and $h^1(\PP^n, I^2_{Z\cap L' \cup L}(d)) = k+  h^1(\PP^n, I^2_L(d))$ or equivalently $h^1(\PP^n,I^2_{Z \cap }{d}) = k$. But as is established in Proposition \ref{h1list}, this just implies that $|Z \cap L| = f(k)$ for some number $f(k)$ depending on $k$. We may then apply Proposition \ref{polyinL} to conclude that  $F_{\lambda,j,l,k}^{M_2}$ is a polynomial in $\Lb$ and it suffices to establish that the image of $\sum_{|\lambda|<N} (-1)^{||\lambda||} \bar F_{\lambda,j,l,k}^{M_2}$ is zero in $K_0(MHS)$.

But by Proposition \ref{pm} the image of  $\sum (-1)^{||\lambda||} \bar F_{\lambda,j,l,k}^{M_i}$ in $K_0(MHS)$ is $$\sum_{i= 0} ^N H_*^c(w_j(L)\times w_{f(k)} (L') \times w_{i- j- f(k)}(\PP^n \setminus L \setminus L') , \pm,\QQ),$$ where in all relevant cases $f(k) \ge \frac{d}{2} $ which since $d$ is assumed to be large is greater than $2$. Thus by Lemma 1.2 in \cite{V} the above cohomology group vanishes and we have the desired equality.

Similarly $$\sum_{|\lambda|< N } (-1)^{||\lambda||} \bar X_{\lambda,j,l,k}^{M_1} =0.$$

   and $$\sum_{|\lambda|< N } (-1)^{||\lambda||} \bar Y_{\lambda,j,l,k}^ =0.$$ This concludes the proof of (7).
\end{proof}
Finally we prove Theorem \ref{mainPn}.
\begin{proof}[Proof of \ref{mainPn}]
     Let $\epsilon > 0$ be the largest possible $\epsilon$ satisfying the conclusions of Propositions \ref{ge4} -\ref{lowdeg}. Let $N = (\frac{3n}{2} + \epsilon) d$. By Proposition \ref{propVW}, $$\frac{U(\PP^n \OO(d))}{H^0(\PP^n, \OO(d)} = \frac {1}{H^0(\PP^n ,\OO(d)}\sum_{|\lambda| \le N} (-1)^{||\lambda||}W_{\ge\lambda} - (-1)^{||\lambda||}W_{\lambda, N + 1}$$ 
     $$= \frac {1}{H^0(\PP^n, \OO(d))} \sum_{l \ge 0}\sum_{|\lambda| \le N}(-1)^{||\lambda||} W^{l,n}_{\ge\lambda} - (-1)^{||\lambda||}W^{l,n}_{\lambda, N + 1}$$
     $$= \frac {1}{H^0(\PP^n, \OO(d))} \sum_{l \ge 1}\sum_{|\lambda| \le N}(-1)^{||\lambda||} (W^{l,n}_{\ge\lambda} - W^{l,n}_{\lambda, N + 1})$$
     
     $$+  \frac {1}{H^0(\PP^n, \OO(d))} \sum_{|\lambda| \le N} (-1)^{||\lambda||} (W^0_{\ge\lambda} - W^0_{\lambda, N + 1}) $$.
    Let us consider the first term 
    $$\frac {1}{H^0(\PP^n, \OO(d))} \sum_{l \ge 1}\sum_{|\lambda| \le N} W^{l,n}_{\ge\lambda} - W^{l,n}_{\lambda, N + 1}$$

    $$ = \sum_{ r \ge 1} \frac {1}{H^0(\PP^n, \OO(d))} \sum_{l \ge 1}\sum_{|\lambda| \le N} (-1)^{||\lambda||}W^{Chow(r),l,n}_{\ge\lambda} - (-1)^{||\lambda||} W^{Chow(r),l,n}_{\lambda, N + 1}$$
     By Proposition \ref{ge4} we may ignore the cases where $r \ge 4$ as the corresponding terms are of dimension $\le (-\frac{3n}{2} - \epsilon)d $. By Proposition \ref{lowdeg}, for $r=1,2$  $$\sum \frac {1}{H^0(X,\LL)} \sum_{l \ge 1}\sum_{|\lambda| \le N} (-1)^{||\lambda||}W^{Chow(r),l,n}_{\ge\lambda} $$ is also  of dimension $\le (-\frac{3n}{2} - \epsilon)d$. By Proposition \ref{geN0} we may ignore the terms $$\sum \frac {1}{H^0(X,\LL)} \sum_{l \ge 1}\sum_{|\lambda| \le N} (-1)^{||\lambda||}W^{Chow(r),l,n}_{\lambda, \ge N + 1} $$ as well.  The remaining term is the term  for $r =3 $, i.e. 
     $$ \frac {1}{H^0(\PP^n,\OO(d))} \sum_{l \ge 1}\sum_{|\lambda| \le N} W^{Chow(3),l,n}_{\ge\lambda} - W^{Chow(3),l,n}_{\lambda ,\ge N  +1}.$$ Proposition \ref{geN0} implies that in   $\le (-\frac{3n}{2} - \epsilon)d $ the above equals 
     $$ \frac {1}{H^0(\PP^n,\OO(d))} \sum_{l \ge 1}\sum_{|\lambda| \le N} W^{Chow(3),l}_{\ge\lambda}.$$

     We will now discuss the other term. $$\frac {1}{H^0(\PP^n, \OO(d))} \sum_{|\lambda| \le N} (-1)^{||\lambda||}W^{0,n}_{\ge\lambda} - W^{0,n}_{\ge  N +1} $$  $$= \frac {1}{H^0(\PP^n ,\OO(d))}\sum _{|\lambda| \le N} (-1)^{||\lambda||} w_{\lambda}^{0,n} \Lb^{-|\lambda|(n+1)} - W^{0,n}_{\lambda , \ge N +1} $$ 

     By Proposition \ref{geN0} the above equals
    $$ =\frac {1}{H^0(\PP^n, \OO(d))}\sum _{|\lambda| \le N} (-1)^{||\lambda||}  \Lb^{-|\lambda|(n+1)} w_{\lambda}- \sum_{r \ge 1 ,l \ge 1} w_{\lambda} ^{Chow(r),l}\Lb^{-|\lambda|(n+1)}$$

    By Proposition \ref{ge4}, we may disregard all terms involving $r \ge 4$ as they contribute in dimension $\le (-\frac{3n}{2} - \epsilon)d $. Similarly, by Proposition \ref{lowdeg}, we may disregard $r=1,2$.  Hence the above equals
    $$ \frac {1}{H^0(\PP^n ,\OO(d))}\sum _{|\lambda| \le N} (-1)^{||\lambda||}  \Lb^{-|\lambda|(n+1)} w_{\lambda}- \sum_{l\ge 1} w_{\lambda} ^{Chow(3),l}.$$ By Proposition \ref{propVW}

    $$\frac {1}{H^0(\PP^n ,\OO(d))}\sum _{|\lambda| \le N} (-1)^{||\lambda||}  \Lb^{-|\lambda|(n+1)} w_{\lambda} = \sum Z_{\PP^n}^{-1}(n+1)$$ up to dimension $(-\frac{3n}{2} - \epsilon)d$.

     Hence $$ \frac {1}{H^0(\PP^n, \OO(d)} \sum_{|\lambda| \le N} (-1)^{||\lambda||}W^0_{\ge\lambda}$$
     
     $$=  Z_{\PP^n}^{-1}(n+1) - \frac{1}{H^0(\PP^n,\OO(d))} \sum_{|\lambda| \le N}\sum_{l \ge 1} (-1)^{||\lambda||}w_{\lambda} ^{Chow(3),l}\Lb^{-n|\lambda|} $$ up to dimension $(-\frac{3n}{2} - \epsilon)d$ .  Note that this latter term (i.e. $$-\frac{1}{H^0(\PP^n,\OO(d))} \sum_{|\lambda| \le N}\sum_{l \ge 1} (-1)^{||\lambda||}w_{\lambda} ^{Chow(3),l}\Lb^{-n|\lambda|}$$) added to the term $$\frac{1}{H^0(\PP^n,\OO(d))} \sum_{|\lambda| \le N}\sum_{l \ge 1} (-1)^{||\lambda||}w_{\lambda} ^{Chow(3),l}\Lb^{-n|\lambda| +l}$$ (which arose earlier) is $Y(d)$.

    Putting this altogether, we have the require equality.



     
\end{proof}

\end{document}